                                     \documentclass{article}  
\usepackage{xcolor}
 \usepackage{multirow}
\usepackage{lipsum}
\usepackage{graphicx}          % Include this line if your 
                                \usepackage{wrapfig}

\def\dbF{\hbox{\rm l\negthinspace F}}

\def\dbP{\hbox{\rm l\negthinspace P}}

\def\dbR{\hbox{\rm l\negthinspace R}}

\def\be{\begin{equation}}
\def\ee{\end{equation}}
\usepackage{amsmath,amssymb,amsthm}
\newtheorem{lemma}{Lemma}

\newtheorem{example}{Example}
\newtheorem{theorem}{Theorem}

\newtheorem{definition}{Definition}
\newtheorem{proposition}{Proposition}

\begin{document}

\title{Risk-Sensitive Mean-Field-Type Games with $L^p-$norm Drifts} % Title, preferably not more 
                                                % than 10 words.

\author{Hamidou Tembine \\ Learning \& Game Theory Lab\\ Division of Engineering  \\ New York University Abu Dhabi}
\date{September 22, 2014}
                                          
\maketitle

\begin{abstract}                          
We study how risk-sensitive players act in situations where the outcome is influenced not only by the state-action profile but also by the distribution of it. In such interactive decision-making problems, the classical mean-field game framework does not apply. We depart from most of the mean-field games  literature by presuming that a decision-maker may include its own-state distribution in its decision.  This leads to the class of mean-field-type games. In mean-field-type situations, a single decision-maker may have a big impact on the mean-field terms for which  new type of optimality equations are derived. We establish a finite dimensional stochastic maximum principle for mean-field-type games where the drift functions have a p-norm structure which weaken the classical Lipschitz and differentiability assumptions. Sufficient optimality equations are established via Dynamic Programming Principle but in infinite dimension. Using de Finetti-Hewitt-Savage theorem, we show  that a propagation of chaos property with virtual particles  holds for the non-linear McKean-Vlasov dynami1cs.   

\end{abstract}

\section{Introduction}

Recently, there has been a renewed interest in optimization and game problems   of mean-field type, where the performance functionals, drifts, diffusions, and jump coefficients depend not only on the state and the control but also on the probability distribution of state-control pair. Most formulations of mean-field type optimization in \cite{b1,b2,b3,hosking,li,alain} have been of risk-neutral type where the performance functionals are the expected values of stage-additive cost functions of Bolza or Mayer type. Not all behavior, however, can be captured by risk-neutral mean-field type optimizations. One way of capturing risk-averse and risk-seeking behaviors is by exponentiating the performance functional before expectation (see \cite{Jac73}). The objective of a risk-sensitive player is then to optimize an exponentiated long-term loss.  The risk-sensitive criterion is related to the robust control via  relative entropic measures. As the risk-sensitive parameters vanish, one gets a risk-neutral maximum principle of mean-field-type. 

\subsection{On Mean-Field Games}
There are several pioneer works on  static and/or  stationary mean-field games. Most of them are under different names such as global games, anonymous games, aggregative games, population games, large games, etc, but share lot of common features. Here we limit ourselves to the  some pioneer works on dynamic mean-field games. One of the first works on mean-field games is \cite{jova82}. Therein, the author proposes a game-theoretic model that explains why smaller firms grow faster and are more likely to fail than larger firms in large economies.
The game is played over a  discrete time space. The mean-field  is the aggregate demand/supply which generates a price dynamics.  The price moves forwardly, and the players react to the price and generate a demand and the firm a supply with associated cost, which regenerates the next price and so on. The author  introduced a system of   backward-forward system to find equilibria (see for example  Section 4, equation D.1 and D.2 in \cite{jova82}). The backward equilibrium equation is obtained as an optimality to the individual response, i.e., the value function associated with the best response to price, and the forward equation for the evolution of price. Therein, the consistency check is about the mean-field of equilibrium actions (population or mass of actions), that is, the equilibrium price solves a fixed-point system: the price regenerated after the reaction of the players through their individual best-responses should be consistent with the price they responded to. Following that analogy,  a more general framework was developed  in \cite{rosenthal1988}, where the mean-field equilibrium is introduced in the content of Markovian dynamic games with large number of decision-makers. A mean-field equilibrium is defined in page 4 of \cite{rosenthal1988} by two conditions:  (i)  each generic player's action is best-response  to the mean-field, and (ii) the mean-field is consistent and is exactly reproduced from the reactions of the players. This matching argument was widely used in the literature as it can be interpreted as a generic player reacting to an evolving mean-field object and at the same time the mean-field is formed from the contributions of all the players. The authors of  \cite{bergin}  show that  show how common noise can be introduced into the mean-field game 
model (so the mean-field distribution evolves stochastically) and extend the Jovanovic-Rosenthal existence theorem.
The methodology developed in \cite{jova82} and  the subsequent series of papers \cite{rosenthal1988,bergin,benamou,benamou2,pierrelouis,peter} share the following assumptions:
\begin{itemize}
\item (Big size) There is a large number of decision-makers, sometimes, infinite, or a continuum of decision-makers.
\item (Anonymity) The index of the  decision-maker does not affect the utility.  
\item (NonAtomicity) A single decision-maker has a negligible effect on the mean-field-term and on the utility.
\end{itemize}

Unfortunately, some of the above conditions appear to be very restrictive in terms of applications, and we explain below how to relax them via mean-field-type game theory.

\subsection{Related works on Mean-Field-Type Game Theory}
\subsubsection*{One decision-maker}
A stochastic maximum principle (SMP) for the risk-sensitive optimal control problems for Markov diffusion processes with an exponential-of-integral performance functional was elegantly derived in \cite{lim} using the relationship between the SMP and the Dynamic Programming Principle (DPP) which expresses the first order adjoint process as the gradient of the value-function of the underlying control problem. This relationship holds only when the value-function is smooth (see Assumption (B4) in \cite{lim}). The approach of \cite{lim} was widely used and extended to jump processes in \cite{shi-1} and \cite{shi-2,recent3},  but still under this smoothness assumption. However, in many cases of interest, the value function is, in the best case, only continuous. Moreover, the relationship between the SMP and the DPP is unclear for non-Markovian dynamics and for mean-field type game problems where the Bellman optimality principle need to be extended. This calls for the need to find a risk-sensitive SMP and DPP for these cases. 
Djehiche et al. (2014, \cite{Boualem2014,alainjoint}) have  established a stochastic maximum principle for  risk-sensitive mean-field-type control where the key mean-field term is the {\it mean state.}   This means that the drift, diffusion, running cost and terminal cost functions  depend on the state, the control and on the  mean of  state. Our work extends the results of \cite{lim} to risk-sensitive control problems for dynamics that are non-Markovian and of mean-field type. One important contribution of \cite{Boualem2014} is that the derivation of the SMP does not require any (explicit) relationship between the first-order adjoint process and a value-function of an underlying control problem.  Using the SMP derived in \cite{hosking}, the approach is easily extended to the case where the mean-field coupling is in terms of the mean of the state and the control processes.
 In \cite{Boualem2014t2}, we have extended the methodology   to risk-sensitive mean-field-type control under partial observation which has interesting applications in risk-sensitive filtering problems including mean-field ensemble Kalman filtering, state tracking and other data assimilation algorithms in geosciences.

\subsubsection*{Two or more  decision-makers}
 The first paper that deals with risk-sensitive  games in a mean-field context is \cite{tembine2014}. Therein, we have derived a verification theorem for a risk-sensitive mean-field game whose underlying dynamics is a Markov diffusion, using a matching argument between a system of Hamilton-Jacobi-Bellman (HJB) equations and the Fokker-Planck equation. This matching argument freezes the mean-field coupling in the dynamics, which yields a  risk-sensitive HJB equation for the value-function. The mean-field coupling is then retrieved through the Fokker-Planck equation satisfied by the marginal law of the optimal state.   The work in \cite{tembine2014} is fundamentally different than the present work. Therein, the mean-field term is frozen to be the equilibrium mean-field term and a single decision-maker cannot influence the mean-field-term. In the present work, we shall show that, when a single decision-maker has a non-negligible effect in the mean-field, the fundamental optimality equations are changed. 
 In \cite{Boualem2014t3} we have analyzed risk-sensitive linear-exponentiated quadratic  games of mean-field-type for which we have provided closed-form expressions using a novel risk-sensitive stochastic maximum principle  derived in  \cite{Boualem2014} which does not use the value function. It allows us, in particular,  to work with the SMP equations in situations where the value function is not  necessarily differentiable.

Substantial progress have been done in the last decade in
mean-field games in the non-cooperative setup. However,
very little is known about cooperative mean-field games.
In \cite{ifac2014} we have introduced cooperative mean-field type games in which the state dynamics and the payoffs depend
not only on the state and actions but also on their probability measure. We establish a time-dependent payoff allocation procedure for coalitions of mean-field type. The
allocated payoff considers not only fairness property but also the cost of making the coalition. Both time consistency and subgame perfectness solution concept equations are established.
\subsection{Mean-field-type games: additional features  }

Risk-sensitive mean-field-type games \cite{automat} are fundamentally   different than risk-sensitive mean-field games. 
In the mean-field game-theoretic  models it is usually assumed that  (i) very large number of players, (ii) players are indistinguishability (in the sense of  the strategies, payoffs, state laws), (iii) individual contribution to the mean-field term is assumed to be negligible.   In {\it mean-field-type games}, none of the assumptions (i)-(iii) is needed. Following \cite{ifac2014}, a mean-field-type game is defined as any game in which the payoff and/or state dynamics involve not only the state and action profiles and also the distribution of the sate-action pair (or its marginals such as distribution of states and distribution of actions). Mean-field-type game theory is suitable for one, two or more players. A typical example is, a single decision-maker with mean-variance payoff. 
In mean-field-type games: (i) a single player can have a big influence on the mean-field term. A typical example is an Air Conditioning system   which  tries to reduce 
the variance of the temperature state with the respect the desired comfort temperature of the user. That the context, there is only one decision-maker, the user, who acts on the controller. The control variable is  between $\{Heating, Cooling, Nothing\}.$  Clearly, the control action has significant impact on the variance of the temperature.
(ii)  there is no need for players to be indistinguishable (see Section \ref{sec:game}).
(iii) there is no need to have large number (or infinite or continuum) of players. 
The mean-field-type game framework allows us to address more interesting real-world applications where the number of decision-makers may be large but still finite  and include both von Neumann and non-von Neumann utility functions.

\subsection{Novelty and Contribution}

Our contribution can be summarized as follows. 
We start with one player risk-sensitive mean-field-type optimization where the state dynamics has $L^p-$norm structure, which is not differentiable. Our main motivation for considering  this structure comes from its applications for the control of virus spread among interactive communities (networks) as observed in \cite{rachev}. This allows us to consider other types of non-linear mean-field  interactions that are not investigated in the literature of mean-field games. It also allow us to consider weakened Lipschitz conditions and non-differentiable drift coefficients. We show that the non-differentiability issue can be handled using weak derivatives or sub-differential set. We derive a stochastic maximum principle and a dual game variable which satisfies the risk-sensitive SMP whenever the associates weak derivatives make sense. In addition, a risk-sensitive DPP is provided in infinite dimension.
We believe that the present paper is the first work that analyzes risk-sensitive mean-field-type games with the $L^p-$norm which is non-differentiable.

\subsection{Structure of the paper}

The paper is organized as follows. In Section \ref{sec:oneplayer}, we present the model and state the main results for one player.  Section \ref{sec:game} presents risk-sensitive mean-field type games with two or more players. We provide a dynamic programming principle in infinite dimension in subSection \ref{dppgame}. Section \ref{sec:virus} focuses on the control of virus spread among interactive communities (networks).
 Section \ref{sec:conclusion} concludes the paper.  For completeness, we provide  in Appendix the de Finetti-Hewitt-Savage theorem and the existence and uniqueness proofs.

To streamline the presentation, we only consider the one-dimensional state case. The extension to the multidimensional case is by now straightforward. The norm is denoted with the  index $\alpha\geq 1$ and $p$ will be used for the adjoint process in the stochastic maximum principle.  Also, it should be noted that our diffusion coefficient is control  independent. More general state, control and mean-field dependent diffusions are  carried out in \cite{Boualem2014}. Also the technique developed here can be easily extended to the jump-diffusion case using the works in \cite{recent1,recent2,Karoui-Ham}.

%%%%%%%%%%%%%%%%%%%%%%%%%%%%%%%%%%%%
\section{Mean-field-type game with one risk-sensitive decision-maker}\label{sec:oneplayer}

 Let $T>0$ be a fixed time horizon, $\alpha\geq 1$ and $(\Omega,{\mathcal{F}},\dbF, \dbP)$  be a given filtered
probability space on which a one-dimensional standard
Brownian motion $B=\{B_s\}_{s\geq0}$ is given, and the filtration $\dbF=\{{\mathcal{F}}_s,\ 0\leq s \leq T\}$ is the natural filtration of $B$ augmented by $\dbP-$null sets of ${\mathcal F}.$
  We consider the following risk-sensitive problem :
\begin{equation}\label{SDEu}
\left\{\begin{array}{lll}
\bar{J}^{\theta}(u(\cdot))\\
=\frac{1}{\theta}\log\left( Ee^{\theta\left[\int_0^Tf(t,x^u(t),m^u(t), u(t))\,dt+ h(x^u(T),m^u(T))\right]}\right),\\
\inf_{u} \bar{J}^{\theta}(u)\ \\
\mbox{subject to}\\
dx^u(t)=\bar{b}(.,t,x^u(t), m^u(t), u(t)) \ dt+ \sigma(.,t,x^u(t))dB(t), \\ x^u(0)=x_0,\
m^u(t):=\mathcal{L}(x^u(t)),
\end{array}\right.
\end{equation}
where  the state space is $\mathcal{X}=\mathbb{R},$   the term $\bar{b}$ is distribution-dependent and
 has the special structure $$\bar{b}=\left(\int_{y\in \mathcal{X}}  |b|^{\alpha}(., t,x^u(t), y, u(t)) m^u(t, dy)\right)^{\frac{1}{\alpha}},$$ i.e., the $L^{\alpha}-$norm of $b$ with the respect to the measure $m^u(t, .).$  
\begin{equation*}
\bar{b}(t,x,m,u) : \ [0,T] \times \mathcal{X}\times\mathcal{P}(\mathcal{X})\times U\longrightarrow \dbR,\ 
\end{equation*}
$t\in[0, T],\ x\in \dbR,\ m\in \mathcal{P}(\mathcal{X}),\ u \in U.$ Notice that for $\alpha>1$ the drift term $\bar{b}$ is non-linear in the measure $m.$
\begin{equation*}
{b}(., t,x,y,u) : \ [0,T] \times \mathcal{X}^2\times U\longrightarrow \dbR,\
\end{equation*}
\begin{equation*}
\sigma(., t,x) : \ [0,T] \times \mathcal{X}\longrightarrow \dbR,\
\end{equation*}
$m^u(t):=\mathcal{L}(x^u(t)):=P_{x^u(t)}$ is the probability law of the random variable $x^u(t).$ 
The parameter $\theta$ is the risk-sensitivity index of the player. The instantaneous cost function is 
\begin{equation*}
f(t,x,m,u): \ [0,T] \times \mathcal{X} \times\mathcal{P}(\mathcal{X}) \times U\longrightarrow \dbR,\
\end{equation*} and the terminal cost function is 
\begin{equation*}
 h(x,m): \,\,\mathcal{X}\times \mathcal{P}(\mathcal{X}) \longrightarrow \dbR.
\end{equation*}
The control strategy $u$ is chosen by the decision-maker. An admissible control strategy $u$ is an $\dbF$-adapted and $L^{ \alpha^*}$-integrable process with values in a non-empty subset $U$ of $\dbR^d$. We denote the set of all admissible strategies of the player by $\mathcal{U}$.

\begin{definition}  %\cite{ifac2014}, 

A mean-field-type game is a game in which the payoff and/or state dynamics involve not only the state-action profiles and also the distribution of the sate-action pair (or its marginals such as distribution of states and distribution of actions).
\end{definition}

\begin{example}
Problem (\ref{SDEu}) is a  mean-field-type game with one decision-maker. The optimality equation of  (\ref{SDEu})  is a nonstandard system from mean-field-type optimal control \cite{alainbook}.

\end{example}
 Given an admissible strategy $u\in\mathcal{U}$ of the decision-maker (player), the state equation in (\ref{SDEu}) is a measure-dependent stochastic differential equation (SDE) with random coefficients. 

In view of (\ref{SDEu}), up to a change of  the parameter $\theta$ into $-\theta$, the optimization of $\bar{J}^{\theta}$ is the same as the following optimization problem 
\begin{equation}\label{rs-cost}
J^{\theta}(u(\cdot))=Ee^{\theta\left[\int_0^Tf(t,x^u(t), m^u(t), u(t))\,dt+ h(x^u(T),m^u(T))\right]},
\end{equation}

 Any $\bar u(\cdot)\in {\mathcal{U}}$ satisfying
\be\label{rs-opt-u}
  J^{\theta}(\bar u(\cdot))=\inf_{u(\cdot)\in {\mathcal{U}}}J^{\theta}(u(\cdot)),
\ee
is called a risk-sensitive optimal strategy. The corresponding state process, solution of the SDE in (\ref{SDEu}), is denoted by $\bar x(\cdot):=x^{\bar u}(\cdot)$.
The mean-field-type optimization problem  that we are interested in,  is to characterize the pair $(\bar x,\bar u)$ solution of the  problem (\ref{SDEu}).  Let $\Psi_T=\int_0^T f(t, x(t),m^u(t), u(t)) dt+h(x(T), m^u(T))$. Then the risk sensitive loss functional is given by 
$$
 \bar{J}^{\theta}=\frac{1}{\theta}\log J^{\theta}   = \frac{1}{\theta}\log \left[ E e^{\theta \Psi_T} \right].
$$
When the risk-sensitive index $\theta$ is small, the loss functional $\bar{J}^{\theta}$ can be expanded as
$$
E[ \Psi_T] +\frac{\theta}{2}\mbox{var}(\Psi_T)+O(\theta^2),
$$
where,  $\mbox{var}(\Psi_T)$ denotes the variance of  $\Psi_T$. If $\theta<0$ , the variance of $\Psi_T$, as a measure of risk,  improves the performance, in which case the optimizer is called {\it risk seeker}. But,  when $\theta>0$,  the variance of $\Psi_T$ worsens the performance $\bar J_{\theta}$,   in which case the optimizer is called {\it risk averse}.
The risk-neutral loss functional $E[ \Psi_{T}]$ can be seen as a limit of risk-sensitive functional $ \bar J_{\theta}$ when $\theta\rightarrow 0$. This is one of the reasons why this criterion attracted lots of attention. The criterion  has also interesting connections with  $H_{\infty}-$mean-field-type optimization. This is easily viewed from the Donsker-Varadhan formula:
\be
\frac{1}{\theta}\log\left(  \int e^{\theta \phi} \ d\nu\right)=\sup_{\mu\in \mathcal{P}(\Omega)}\left[ \int {\phi} \ d\mu -\frac{1}{\theta} \tilde{H}(\mu | \ \nu) \right],
\ee
for any measurable bounded function $\phi$ on $\Omega,$  and $\nu$ a probability measure on $\Omega.$ Moreover, the supremum is {\it uniquely} achieved by the {\it imitative } Boltzmann-Gibbs distribution $\mu_*$ widely used in distributed strategic learning \cite{tembine2012}, 
$$
d\mu_*= \frac{e^{\theta \phi}\ d\nu }{\int_{\Omega}  e^{\theta \phi}\ d\nu}.
$$
The function $\tilde{H}(.|.)$ is the relative entropy from $\Omega $ to $\bar{\mathbb{R}}$ given by
$$
\tilde{H}(\mu | \nu):=\int \log \left( \frac{d\mu }{ d\nu}\right) d\mu=\int \frac{d\mu }{ d\nu} \log\left( \frac{d\mu }{ d\nu}\right)\ d\nu,
$$
whenever $\mu\in \mathcal{P}(\Omega)$ is absolutely continuous with the respect to $\nu,$ otherwise we set $\tilde{H}(\mu | \nu)=+\infty.$ The problem is 
\be
\inf_{u} \bar{J}^{\theta}(u) =\inf_{u} \sup_{\mu\in \mathcal{P}(\Omega)}\left[ \mathbb{E}_{\mu}[\Psi_T]  -\frac{1}{\theta} \tilde{H}(\mu | \ \mathbb{P}) \right],
\ee

\subsection{Existence of solution to the state equation}
We now focus on the well-posedness of the state dynamics.

\begin{proposition}\label{cond1} If the functions $b$ and $ \sigma$ are Lipschitz  with the respect to $(x,y)$ with Lipschitz constant $L>0$ and 
 $$\int_0^T \left(\int_y |b(t, 0,y,u)|^{ \alpha} m(t,dy)\ \right)^{1/ \alpha}dt < +\infty,$$ and then,
 the SDE in (\ref{SDEu}) admits a unique strong solution $x^u$ in $L^{\alpha}.$ If in addition, $$\int_0^T \left(\int_y |b(t, X,y,u)|^{2 \alpha} m(t,dy)\right)\ dt < +\infty,$$ a.s. then $\forall \ \alpha\geq 2,\ $
$$
  \sup_n \ \sqrt{n}\ \{ \mathbb{E}\left(\sup_{t\leq T}\ | x_{i,n}(t)- \bar{x}_{i,n}(t)| ^{ \alpha} \right)\}^{\frac{1}{ \alpha}} < +\infty,
$$
where $x_{i,n}(t)$ a $i-$th particle state solution of 
\begin{equation}\label{SDEuram1}
\left\{\begin{array}{lll}
dx^u_{i,n}(t)=\left(  \frac{1}{n}\sum_{j=1}^n  | b(., t,x^u_{i,n}(t), x^u_{j,n}(t), u(t)) |^{ \alpha}              \right)^{\frac{1}{ \alpha}}dt \\
+ \sigma(., t,x^u_{i,n}(t))dB_{i,n}(t), \\ \ \mathcal{L}(x^u_{i,n}(0)) = m_0,\\
\end{array}\right.
\end{equation} and $ \bar{x}_{i,n}(t)$ has the law of  ${x}^{u}(t).$
\end{proposition}
%%%put it at the end

\begin{proof}
See Appendix.
\end{proof}
As we provide in Theorem \ref{theoremsavage} in Appendix, the mean-field convergence of the empirical measure  $\frac{1}{n}\sum_{i=1}^n \delta_{x^u_{i,n}(t)} $ is by now  a well-established result under 
de Finetti-Hewitt-Savage theorem. The issue here is to identify the limiting measure  $m$ with the particularity of  the $L^{\alpha}-$norm structure.  We provide an example of mean-field-type SDE in cooperative dynamics.
\begin{example}[Effect of mean-field in cooperative dynamics]
Consider the   mean-field  stochastic dynamics 
with drift $$\left\{\int \|\sin(-x^3(t)+x(t))-\mu\sin(x(t)-y)\|^{ \alpha} m(t,dy)\right\}^{1/ \alpha},$$ 
where $\mu>0,$ and  constant diffusion coefficient $\sigma\in \mathbb{R}.$ The first term in the drift ($\sin(-x^3(t)+x(t))$) is often replaced by a control action $u(t)\in [-1,1]$ to get
\begin{eqnarray}
dx(t)&=&\sigma dB(t)+\nonumber \\ &&
\left\{\int \| u(t)-\mu \sin(x(t)-y)\|^{ \alpha} m(t,dy)\right\}^{1/ \alpha}dt
\end{eqnarray}
This type of mean-field SDE models has been used to understand muscle contraction (see Section 5 in \cite{bio1}). Other similar models have been widely studied in chemical kinetics, statistical mechanics and economics to capture cooperative behavior of a generic particle, oscillator or an agent. 
\end{example}

 Note that the presence of the measures   $m^{u}(s), \ 0\le s\le T$, in the loss function  $J^{\theta}$  may cause  time-inconsistency, in which case the Bellman's Principle  with the state $x$ is no longer valid and this motivates the use of the stochastic maximum principle (SMP) approach to a get a finite dimensional framework. Note, however that, one can apply DPP where  the state in infinite dimension $\mu(t,.)$ as shown in Section \ref{dppgame} (see  also \cite{dppnew}). 

\subsection{Stochastic Maximum Principle}
We define the risk-neutral Hamiltonian associated with random variables $X\in L^{ \alpha}(\Omega, {\mathcal{F}},\dbP)$ as follows.  for $(p, q)\in \dbR\times \dbR$
\begin{eqnarray}\label{hamiltonian}
H(t,X,m,u,p,q):=\\ \bar{b}(t,X,m,u)p+\sigma(t,X)(t,X)q-f(t,X,m,u).\nonumber
\end{eqnarray}

We also introduce the risk-sensitive Hamiltonian: for $\theta \in \dbR$ and $(p,q, \ell)\in \dbR\times \dbR\times \dbR$,
\begin{equation}\label{rs-hamiltonian}
H^{\theta}(t, X,m, u, p, q,\ell):= \bar{b} p+\sigma(q +\theta\ell p)-f.
\end{equation}
We have $H=H^0.$ The sign $(-f)$ is used here for the only purpose of having maximum principle instead of minimum principle and does not fundamentally change the methodology. 
Moreover, we  denote
\be\label{dH}\begin{array}{lll}
 H_k^{\theta}(t):=p(t)\bar{b}_k(t)+(q +\theta\ell p) \sigma_k(t)- f_k(t),
\end{array}
\ee
for $k\in \{x,m\}.$

Note that even if $b$ is differentiable, the drift coefficient $\bar{b}$ which is $$\ \left(\int_{y\in \mathcal{X}}  b^{ \alpha}(., t,x^u(t), y, u(t)) m^u(t, dy)\right)^{\frac{1}{ \alpha}},$$ may not be  differentiable at the points where $b(.)=0.$
Denote by 
$$\bar{b}_x(t)=
\frac{\int_{y\in \mathcal{X}}  b_x b^{ \alpha-1}(., t,x^u(t), y, u(t)) m^u(t, dy) } {  \bar{b}^{ \alpha-1}}
$$ if $\bar{b}^{ \alpha-1}(x,m)> 0.$ The case where $\bar{b}^{ \alpha-1}(x,m)<0$ is handled in a similar way.
The differentiation with the  respect to the measure $m$ is considered in a Gateaux-derivative sense as in \cite{alainbook}.
$$
\lim_{\epsilon \rightarrow 0_+}\frac{d}{d\epsilon} \bar{b}(.,t, x, m+\epsilon d)=\int   \bar{b}_m(.,t,x,m)(\xi) \ d(d\xi).
$$

\begin{example} We provide Gateaux differentiation of $\| x\|_{\alpha}-$based functions:
\begin{itemize}\item Mean state: 
Let $f(., t, x,m)=\int y m(t, dy).$  Then,
$$
\lim_{\epsilon \rightarrow 0_+}\frac{d}{d\epsilon} f(.,t, x, m+\epsilon d)=\int   \xi  \ d(d\xi).
$$
The Gateaux-derivative with the respect to $m$ is
$
f_m(., t, x,m)(\xi)=\xi.
$
Then $
f_m(., t, \xi,m)(x)=x,\   \partial_{x}[f_m(., t, \xi,m)(x)]=1.
$
It is therefore clear that  $\partial_{x}[f_m(., t, \xi,m)(x)]=1\neq 0=\partial_{m}[f_x].$
\item Square of the mean:

Let $g(., t, x,m)=\frac{1}{2}(\int y m(t, dy))^2.$  Then,
$$
\lim_{\epsilon \rightarrow 0_+}\frac{d}{d\epsilon} g(.,t, x, m+\epsilon d)=\bar{m}\int   \xi  \ d(d\xi).
$$
$g_m(., t, x,m)(\xi)=\xi \bar{m}. $ Hence,
$g_m(., t, \xi,m)(x)=x \bar{m}, $ and
$\partial_x g_m(., t, \tilde{X},m)(x)= \bar{m}. $
\item Second moment: 
If  $g(., t, x,m)=\frac{1}{2}\int y^2 m(t, dy)$  then,
$$
\lim_{\epsilon \rightarrow 0_+}\frac{d}{d\epsilon} g(.,t, x, m+\epsilon d)=\frac{1}{2}\int   \xi^2  \ d(d\xi).
$$
$g_m(., t, x,m)(\xi)=\frac{1}{2}\xi^2 . $ Hence,
$g_m(., t, \xi,m)(x)=\frac{1}{2}x^2. $
$\partial_x g_m(., t, \tilde{X},m)(x)= x. $
\item $\alpha-$th moment:  
$g(., t, x,m)=\int y^{\alpha} m(t, dy)$  then,
$
\lim_{\epsilon \rightarrow 0_+}\frac{d}{d\epsilon} g(.,t, x, m+\epsilon d)=\int   \xi^{ \alpha} \ d(d\xi).
$
Thus, $g_m(., t, \xi,m)(x)= x^{ \alpha},$ and
$\partial_x g_m(., t, \tilde{X},m)(x)=  \alpha x^{ \alpha-1}. $

\item   $\alpha-$norm: $g(., t, x,m)=(\int |y|^{ \alpha} m(t, dy))^{1/  \alpha}=m_{\alpha}^{1/  \alpha}$  then,
$$
\lim_{\epsilon \rightarrow 0_+}\frac{d}{d\epsilon} g(.,t, x, m+\epsilon d)$$ $$=\frac{1}{ \alpha} \left[\int   |\xi|^{ \alpha}  \ m(d\xi)\right]^{\frac{1}{ \alpha}-1} \left[\int   |\xi|^{ \alpha}  \ d(d\xi)\right].
$$
Thus, $g_m(., t, \xi,m)(x)= \frac{x^{ \alpha}}{ \alpha m_{ \alpha}^{ \alpha-1}},$ and
$\partial_x g_m(., t, \tilde{X},m)(x)= \frac{x^{ \alpha-1}}{m_{ \alpha}^{ \alpha-1}}. $
\item $L^{ \alpha}-$normed drift:
We compute the Gateaux-derivative of the $L^{ \alpha}-$normed drifts: $ \bar{b}_m(.,t,x,m)(\xi):= \frac{b^{ \alpha}(.,t,x,\xi)}{ \alpha \bar{b}^{ \alpha-1}}.$ By changing variables, one has
$ \bar{b}_m(.,t,\xi,m)(x):= \frac{b^{ \alpha}(.,t,\xi,x)}{ \alpha \bar{b}^{ \alpha-1}(t,\xi,m)}.$
We differentiate with the respect to $x$ to get:
\begin{eqnarray} \partial_x \bar{b}_m(.,t,\xi,m)(x)
 &=& \frac{b^{ \alpha-1}(.,t,\xi,x) b_y(.,t,\xi,x)}{\bar{b}^{ \alpha-1}(t,\xi,m)}.
 \end{eqnarray}

$E[L \partial_x \bar{b}_m(.,t,X,m)(x)]= E[L \frac{b^{ \alpha-1}(.,t,X,x) b_y(.,t,X,x)}{\bar{b}^{ \alpha-1}(t,X,m)}]:=\tilde{E}[\tilde{L} \frac{b^{ \alpha-1}(.,t,\tilde{X},x) b_y(.,t,\tilde{X},x)}{\bar{b}^{ \alpha-1}(t,\tilde{X},m)}],$ where the notation $\tilde{E}$ denotes the expectation with the respect to the variables with $\tilde{X}$ which is an  copy of $X.$
We now replace the argument $x$ by $X$ to get
$\tilde{E}[\tilde{L} \partial_x \bar{b}_m(.,t,\tilde{X},m)(X)]=\tilde{E}[\tilde{L} \frac{b^{ \alpha-1}(.,t,\tilde{X},X) b_y(.,t,\tilde{X},X)}{\bar{b}^{ \alpha-1}(t,\tilde{X},m)}].$
If $ \alpha=1,$ one  gets  $\tilde{E}[\tilde{L} b_y(.,t,\tilde{X},X)].$
\end{itemize}
\end{example}

We now  introduce the first  order adjoint processes involved in the risk-sensitive SMP.
The (risk-sensitive) first order adjoint equation is the following  backward SDE of mean-field type:
\be\label{rs-firstAD-1}\left\{\begin{array}{lll}
d {p}(t)= -\left\{H^{\theta}_x(t)+\frac{1}{v^{\theta}(t)}E[v^{\theta}(t) \partial_xH^{\theta}_{m}(t)]\right\} dt \\ + {q}(t)(-\theta\ell(t)dt+ d B_t),
\\
dv^{\theta}(t)=\theta\ell(t)v^{\theta}(t)dB_t,
\\
v^{\theta}(T)= \phi^{\theta}(T),\\
{p}(T)= -h_x(T)-\frac{1}{\phi^{\theta}(T)}E[\phi^{\theta}(T) \partial_xh_{m}(T)].
\end{array}
\right.
\ee
where,
\be\label{phi}
\phi^{\theta}(T):=e^{\theta [h(\bar x(T), \bar m(T))+ \int_0^Tf(t,\bar x(t), \bar m(t), \bar u(t)) dt]}. 
\ee

{Note that the  Hamiltonian terms in (\ref{rs-firstAD-1}) are evaluated at the optimal state and optimal control $(\bar x(\cdot),\bar u(\cdot)),$ i.e., 
$H^{\theta}_k(t, \bar{x}(t),\bar{m}(t), \bar{u}(t), \bar{p}(t), \bar{q}(t),\ell(t)), \  \ k\in \{ x,m\}.$
}

\begin{lemma}[\cite{Buckdahn1,b3}] \label{lemt1}
Consider the following mean-field backward SDE
$$
p(t)=p(T)+\int_t^T \tilde{E}[\hat{f}(s, \tilde{p}(s), \tilde{q}(s), p(s), q(s))]\ ds $$ $$ -\int_t^T q(s) dB(s),\ 
$$
where $p(T)$ is a progressively measurable, square integrable random variable. 
Let $\hat{f}(t, .,.,.,.)$ be Lipschitz for all time $t\in [0,T]$ and $t \mapsto f(t, 0,0,0,0)$ be square integrable over $[0,T].$
Then, the mean-field backward SDE has a unique adapted solution satisfying
\begin{equation}\label{rs-second-boundst2}
E\left[\sup_{t\in[0,T]}| p(t)|^2+\int_0^T | q(t)|^2 dt\right]<\infty.
\end{equation} 
Note that, by choosing $\hat{f}(t, \tilde{p}(t), \tilde{q}(t), p(t), q(t))=a_0(t,.)+a_1(t,.) \tilde{p}(t)+a_2(t,.) \tilde{q}(t)+a_3(t,.) p +a_4(t,.)q(t)$ where $a_i(t,.)$ are measurable bounded coefficient functions, one gets a backward equation in the form of  the adjoint equations.  
\end{lemma}

\begin{proposition}\label{EU-rs-first}
If the functions $b, \sigma, f, h,$ are twice continuously differentiable with respect to $(x,m)$ and 
 $b, \sigma, f,h$ and all their first order derivatives with respect to $(x,m)$ are continuous in $(x,m, u)$, and bounded then
(\ref{rs-firstAD-1}) admits an $\dbF$-adapted solution $(\bar p,\bar q,v^{\theta},\ell)$ such that
\begin{eqnarray} \nonumber
E\left[\sup_{t\in[0,T]}|\bar p(t)|^2+\sup_{t\in[0,T]}|v^{\theta}(t)|^2\right.\\ \label{rs-first-bounds}
\left. +\int_0^T \left(|\bar q(t)|^2+|\ell(t)|^2\right) dt\right]<\infty. 
\end{eqnarray} 

In addition, if $b>0$ then (\ref{rs-firstAD-1}) admits a unique  $\dbF$-adapted solution. 
\end{proposition}

 \begin{proof}
 Under the assumptions of  Proposition \ref{EU-rs-first}, the processes $p,q$ solve a backward SDE coupled with the process $v^{\theta}$. Moreover, these equations can be transformed into linear SDEs of mean-field-type, which involves $(p,q, E[p], E[q])$ and random coefficients involving $(x, L^{\theta}:=\frac{v^{\theta}}{E[\phi^{\theta}(T)]},l).$ We now check that the coefficients of the linear SDEs does not blow-up within the horizon $[0,T].$ For the functions  $\sigma, f, h,$  and their derivatives the boundedness follow from the assumption. However, it is not immediate for the drift coefficient $\bar{b}$ . We recall that $\bar{b}_x$, $\bar{b}_m$, $\bar{b}_{xm}$ are not clearly defined 
 at the point where $b(.,t,x,.)$ is  zero. Since $\alpha\geq 1,$ we replace these terms by any representation in the sub-differential set. All terms are bounded by $M=\sup_{t\in [0,T]}\sup  |b_x(t,.)|.$ The process $v^{\theta}(t)$ is almost surely bounded. Then,  for each direction chosen in the sub-differential, the assumptions of Lemma  \ref{lemt1} are fulfilled and hence,  the existence  of solution to the first order risk-sensitive adjoint equations follows. Moreover, if $b>0$ then the denominator does not vanish and  $\bar{b}_x$ and $\partial_x\bar{b}_{m}$ are (uniquely) well-defined, and bounded by $M.$ Using Lemma   \ref{lemt1} again we get existence and uniqueness of solution.
 \end{proof} 
Note that the boundedness and differentiation conditions can be weaken by using the techniques developed in \cite{Mezerdi,jourdain}.
 The following Proposition is the stochastic maximum principle for  Problem (\ref{SDEu}).

\begin{proposition}\label{main result}$(${\bf Risk-sensitive maximum principle}$)$
Let  the Assumptions of Proposition \ref{EU-rs-first}  hold. If  $(\bar x(\cdot),\bar  u(\cdot))$  is an optimal solution of the risk-sensitive control problem (\ref{SDEu})-(\ref{rs-opt-u}), then there are two pairs of $\dbF$-adapted processes $(v^{\theta}, \ell)$, $(\bar p,\bar{q})$  that satisfy (\ref{rs-firstAD-1})-(\ref{rs-first-bounds})  respectively, such that

$$
H^{\theta}(t, \bar x(t),\bar m(t),\bar u(t), \bar{p}(t), \bar{q}(t),\ell(t))$$ $$=\max_{u} H^{\theta}(t,\bar x(t),  \bar m(t), u, \bar{p}(t), \bar{q}(t),\ell(t)),
$$ for almost every $t\in [0,T]$ and $\mathbb{P}-$almost surely.
\end{proposition}

 \begin{proof}
 To prove the SMP, we use  a logarithmic transformation and follows similar steps as in \cite{Boualem2014}.
 \end{proof}
Below we provide an explicit representation of the process of the SMP via a dual approach and partial differential equations of mean-field type.
\section{ Mean-Field-Type Games: two or more risk-sensitive players} \label{sec:game}
We now consider two or more risk-sensitive players. The risk-sensitivity index of player $i$ is $\theta_i.$ The best response to $u_{-i}, m$ is the following  problem :
\begin{equation}\label{SDEugame}
\left\{\begin{array}{lll}
\bar{J}^{\theta}_i(u_i(\cdot), u_{-i}(\cdot))=\\ \frac{1}{\theta_i}\log\left( Ee^{\theta_i\left[\int_0^Tf_i(t,x^u(t),m^u(t), u(t))\,dt+ h_i(x^u(T),m^u(T))\right]}\right),\\
\inf_{u_i} \bar{J}^{\theta}_i(u_i, u_{-i})\ \\
\mbox{subject to}\\
dx^u(t)=\bar{b}(t, x^u(t),m^u(t),u(t))  \ dt+ \sigma( t,x^u(t))dB(t), \\ x^u(0)=x_0,\\
m^u(t):=\mathcal{L}(x^u(t)),
\end{array}\right.
\end{equation}
where $u_{-i}$ denotes $(u_1,\ldots, u_{i}, u_{i+1},\ldots, u_n), \ n\geq 2.$ and by abuse of notation, $u=(u_i,u_{-i}).$
Note that we cannot  impose indistinguishability of the players since $\theta_i$ and the objectives $f_i, h_i$ may be different across the players.

\subsection{Main Result}
We now present the key  results of the paper.
The risk-sensitive game with cost ${J}^{\theta}_i$ solves a system of 
 risk-sensitive HJB  equations
\begin{eqnarray}  \nonumber 0&=&V_{i,t}(\mu)+\int \inf_{u}\left[ \bar{b} \partial_xV_{i,\mu}(\mu)(t,x,z)+f\partial_zV_{i,\mu}(\mu)(t,x,z)       \right. \\ \label{hjbrst2} &&
\left.   +\frac{\sigma^2}{2} \partial_{xx}V_{i,\mu}(\mu)(t,x,z) \right] \mu(t,dx,dz) , \end{eqnarray}
which is a partial differential equation with state $\mu$ (in infinite dimension).
If  we denote $v^*_i(t,x,z):=V_{i,\mu}(t,\mu)(t,x,z)$ as a dual function (because of the Gateaux derivative with the respect  to $\mu$) then  $(v^*_1, \ldots, v^*_n)$ is in finite dimension and solves  the dual system
\begin{eqnarray}  \label{adjointeq1rst2}
0&=&v^*_{i,t}(t,x,z)+v^*_{i,z}(t,x,z) {H}^*_i(t, x, \frac{v^*_{i,x}(t,x,z)}{v^*_{i,z}(t,x,z)}, m)\\  \nonumber &&+\frac{1}{2}\sigma^2 v^*_{i,xx}(t,x,z)+\\
&& \int_{\tilde{w}} v^*_{i,z}(t,\tilde{w})H^*_{i,m}(t, \tilde{x}, \frac{v^*_{i,x}(t,\tilde{w})}{v^*_{i,z}(t,\tilde{w})}, m) (x)\ \mu(t,d\tilde{w}) \nonumber \\
&& i\in \{ 1,2,\ldots,n \} \nonumber
\end{eqnarray}
where $H_i$ is the Hamiltonian, $u_i\in \arg\min H_i,$  $\tilde{w}=(\tilde{x},\tilde{z}),$
and $m(t,x)=\int \mu(t,x,dz_1\ldots dz_n),$ and $\mu(t,.)$ solves the Kolmogorov equation in which $u$ is replaced by the optimal strategies $(u_1,u_2,\ldots, u_n).$
Then $(p^*_i, q^*_i, \eta^*_i, l_i) $ solves the  (risk-sensitive) stochastic maximum principle  system given by:

\be\label{rs-firstAD-1rsmaint2}\left\{\begin{array}{lll}
dp^*_i=-\left[ H^*_{i,x}  +\frac{1}{\eta^*_i} E [ \eta_i^* \partial_xH^*_{i,m}] +  \sigma_x(\frac{\sigma v^*_{i,xx}}{\eta^*_i}) \right. \\
\left.- p^*_i( l^*_i)^2+\frac{\sigma l^*_i v^*_{i,xx}}{\eta^*_i}\right] + (- p^*_i l^*_i+\frac{\sigma v^*_{i,xx}}{\eta^*_i}) dB_i
\\
 d\eta^*_i= \eta^*_i l^*_i dB_i,\\
  \eta^*_i(T)=\theta_i e^{\theta_i[z_i(T)+h_i(x(T), m(T))]},\\
  q_i=(- p^*_i l^*_i+\frac{\sigma v^*_{i,xx}}{\eta^*_i})
\\
p^*_i(T)= h_{i,x}(T)+\frac{1}{\phi^{\theta}_i(T)}E[\phi^{\theta}_i(T) \partial_x h_{i,m}(T)],\\
i\in \{1,2,\ldots, n\}
\end{array}
\right.
\ee
where,
\be\label{phirst2}
\phi^{\theta}_i(T):=e^{\theta_i [h_i(\bar x(T), \bar m(T))+ \int_0^Tf_i(t,\bar x(t), \bar m(t), \bar u(t)) dt]}. 
\ee
%%%%%%%%%%%%%%%%%%%%

 \subsection{Dynamic programming for risk-sensitive mean-field-type games} \label{dppgame}
We establish a dynamic programming principle in infinite dimension. We first write the objectives as a function of the infinite dimensional state $\mu$ which satisfies  the Fokker-Planck-Kolmogorov forward equation
\be  \label{fpk2}
\mu_t=-\partial_x[ \bar{b} \mu] -\partial_z(f \mu)+\frac{1}{2}\partial_{xx}(\sigma^2 \mu)=:\tilde{b},\ \ 
\ee with the initial distribution $\ \mu(0,dx,dz)=m_0(dx)\delta_0(dz).$
The advantage now is that $\mu(.)$ is a deterministic object  
the cost can be rewritten in a deterministic manner as
$$J^{\theta}_i=\int \mu(T, dx,dz) \ e^{\theta_i(z_i+h_i(x, \int_{\tilde{z}}\mu(T,.,d\tilde{z})))}\ .$$ This a terminal cost in the sense it is evaluated only at $\mu(T,.)$. Since there is no running cost, one can write directly the HJB equation using classical calculus of variations for
$$V_i(t,\mu(t,.))=\inf_{u}  \int \mu(T, dx,dz) \ e^{\theta_i(z_i+h_i(x, \int_{\tilde{z}}\mu(T,.,d\tilde{z})))} ,$$ starting from  $\mu(t,.)$ at time $t:$

$$0=V_{i,t}+\inf_{u}[ \langle \tilde{b}, V_{i,\mu} \rangle ],$$ where
\begin{eqnarray}\langle \tilde{b}, V_{i,\mu} \rangle &=&\int  V_{i,\mu} (\mu)(\tilde{x},\tilde{z})\tilde{b}(\tilde{x},\tilde{z}) \ d\tilde{x}d\tilde{z}\\
&=& -\int  V_{i,\mu} (\mu)({x},{z})   \{\partial_x[ \bar{b} \mu] +\partial_z(f_i\mu) \}    \\ &&   +\int  V_{i,\mu} (\mu)({x},{z}) \frac{1}{2}\partial_{xx}(\sigma^2 \mu)                    
\ d{x}d{z}\\
&=& \int \left[  \bar{b}\partial_xV_{i,\mu}     +f_i\partial_zV_{i,\mu}        +\frac{\sigma^2}{2} \partial_{xx}V_{i,\mu}  \right] \mu(t,dx,dz)    \nonumber               
\end{eqnarray}
As we can see, the required working state for player $i$ is $(x,z_i),$ therefore the partial derivatives of  $V_i$ with respect to $z$ are only considered for $z_i.$
The risk-sensitive HJB minimum principle yields
\begin{eqnarray} \nonumber 0=V_{i,t}+\int \inf_{u}\left[ \bar{b} \partial_xV_{i,\mu}      +f_i\partial_zV_{i,\mu} \right. \\  \label{hjbrs} \left. +\frac{\sigma^2}{2} \partial_{xx}V_{i,\mu}  \right] \mu(t,dx,dz) ,\end{eqnarray} 
This is an infinite dimensional PDE on $(t,\mu).$ Below we provide a simpler optimality equation (i.e., the state will be in finite dimension) by setting
$p^*_i=\frac{\partial_xV_{i,\mu}}{\partial_zV_{i,\mu}}$ and $V_{i,\mu}(\mu)(t,x,z)=v^*_i(t,x,z).$
Differentiating (\ref{hjbrs}) with the respect to $\mu$ one gets
\begin{eqnarray}  \label{hjbrs2}0=\partial_tV_{i,\mu}(\mu)(t,x,z)+ \hat{H}_{i}(\mu)\\  \nonumber+\int \hat{H}_{i,\mu}(\mu)(t,\tilde{x},\tilde{z})\mu(t,d\tilde{x},d\tilde{z}) \end{eqnarray}
where 
\begin{eqnarray}
\hat{H}_i&=& \bar{b} \partial_xV_{i,\mu}      +f_i\partial_zV_{i,\mu}   +\frac{\sigma^2}{2} \partial_{xx}V_{i,\mu} 
\end{eqnarray} 
where $
{H}^*_i(t,x,p^*,m)= \inf_{u}[\bar{b} p^*_i+ f_i].
$
\begin{definition} The function 
$v^*_i(t,x,z):=V_{i,\mu}(\mu)(t,x,z)$ is called Dual Function associated with the best response value of player $i.$
\end{definition}
$v^*_i(t,x,z):=V_{i,\mu}(\mu)(t,x,z)$ solves the  PDE (\ref{adjointeq1rst2}) where the state is now reduced to $(x,z)$ which is in finite dimension.

Below we show that if there exists a  dual function (in the sense of weak derivatives) then its weak derivatives provide a risk-sensitive SMP.
\subsection{Dual functions  associated with the best response values }
In a risk-neutral setting, Bensoussan et al.  have established in \cite{alainbook} a partial differential equation as a 
 necessary condition for optimality under smoothness assumption. We apply the methodology to the risk-sensitive case.
The basic idea consists to write the optimality inequality as $J^{\theta}_i(u_i+ \epsilon d_i, u_{-i} ,m^{u_i+ \epsilon d_i,u_{-i}})- J^{\theta}_i(u,  m^u) \geq 0$ for the cost functional  $J^{\theta}.$ 
By introducing the auxiliary state $z$ such that
$ dz_i= f_i(.) \ dt,\  \ z_i(0)=0,$ the risk-sensitive game problem is transformed into mean-field-type game problem without running cost. The terminal cost is
$
e^{\theta_i [z_i(T)+h_i(x(T),m(T))]}.
$
Since the state is augmented to be $(x,z),$ the first order adjoint process becomes $(p_{1i}, p_{2i}, q_i)$ and the unmaximized Hamiltonian is $\bar{b} p_{1i}+ f {p}_{2i} +\sigma q_i.$ Since the diffusion does not depend on the control $u,$ the derivative of this term with the respect to $u$ can be written as
$
{p}_{2i} \partial_u [ \bar{b} \frac{p_{1i}}{{p}_{2i}}+ f_i], \ 
$
whenever $ {p}_{2i}\neq 0.$ 

Let $v^*_i(t,x,z)$ be the dual function defined above, satisfying  (\ref{adjointeq1rst2})
where  $\mu=\mathcal{L}(x^{u}(t),z^u(t)),$  the $x-$marginal  is $ m(t,.)=\int  \mu(t,.,dz)\ $  and $$
{H}_i^*(t,x,p^*,m)= \inf_{u}[\bar{b} p^*_i+ f_i].
$$ 
The terminal condition is
$$v^*_i(T,x,z)=e^{\theta_i (z_i+h_i(x,m(T)))} $$ $$+\theta_i \int h_{i,m}(\tilde{x},m(T))(x)  e^{\theta_i (\tilde{z}_i+h_i(\tilde{x},m(T)))}\ \mu(T, d\tilde{x},d\tilde{z}).
$$

If $(v^*_i)_i$ solves the dual equation (\ref{adjointeq1rst2}) then
$$
\lim_{\epsilon  \rightarrow  0_+}\frac{d}{d\epsilon}J^{\theta}_i=$$ $$\int_{(t,x)\in [0,T]\times \mathcal{X}}\  v^*_{i,z} H_{i,u}(t,x,m,u, \frac{v^*_{i,x}}{ v^*_{i,z}}, \frac{\sigma v^*_{i,xx}}{ 2v^*_{i,z}})\ \mu(t,x,z) dx dt
$$
and ${H}_{i,u}=0$ for interior  optimal control $u.$
Let $v_i$  be a function in the Lebesgue space $L^1(I)$, with $I=[a,b]$ a compact interval of $\mathbb{R}, \ a<b.$ We say that $w\in L^1(I)$ is a ''weak derivative'' of  $v$ if, 
$$\int_I v(t)\varphi'(t)dt=-\int_I w(t)\varphi(t)dt, $$
for all infinitely differentiable functions $\varphi $ with $\varphi(a)=\varphi(b)=0.$

Equation (\ref{adjointeq1rst2}) is an interesting partial differential equation. Indeed, if there is a solution $v^*(t,x,z)$ to (\ref{adjointeq1rst2}) that is three times weakly differentiable  then the partial weak derivatives of $ v^*(t,x,z)$ solves the risk-sensitive SMP (\ref{rs-firstAD-1}). 
Below we identify explicitly  the processes solution to the  risk-sensitive SMP.
Let $p^*_i(t)=  \frac{v^*_{i,x}(t,x^{u}(t),z^u(t))}{v^*_{i,z}(t,x^{u}(t),z^u(t))}= \frac{\partial_xV_{i,\mu}(t,x^{u}(t),z^u(t))}{\partial_zV_{i,\mu}(t,x^{u}(t),z^u(t))}$ where the derivatives are taken in a distribution sense (weak derivative).
Then the process $p^*$ evaluated at the optimal trajectory solves the backward SDE:
\begin{eqnarray}
dp^*_i&=&-\left[ H^*_{i,x}  +\frac{1}{\eta^*_i} E [ \eta^*_i \partial_xH^*_{i,m}] \right]dt \nonumber \\
&& - \left[  \sigma_x(q^*_i+ p^*_i l^*_i)+ q^*_i l^*_i\right] dt+ q^*_i dB,
\end{eqnarray}
where
$$
q^*_i=- p^*_i l^*_i+\frac{\sigma v^*_{i,xx}}{\eta^*_i},\ \ l^*_i=\sigma  \partial_x[\log \eta^*_i], $$ $$  \eta^*_i:=v^*_{i,z}(t,x,z)=\partial_z V_{i,\mu}(\mu)(t,x,z),
$$

We impose a strong smoothness on $v^*$ and show that $\eta^*$ solves a backward SDE similar to the one satisfied by $v^{\theta}.$
\begin{proposition} \label{lemmatx1}
$\eta^*_i:=v^*_{i,z}(t,x,z)$ solves the  backward SDE: 
\be d\eta^*_i= \eta^*_i l^*_i dB,\ \ \ \  \eta^*_i(T)=\theta_i e^{\theta_i[z_i(T)+h_i(x(T), m(T))]}. \ee
\end{proposition}
%%%
\begin{proof} 
By Ito's formula, we have
\begin{eqnarray}d\eta^*_i&:=& dv^*_{iz}(t,x,z)\\ \nonumber
&=& [v^*_{i,zt}+v^*_{i,xz} \bar{b}+v^*_{i,zz} f_i+\frac{1}{2}\sigma^2v^*_{i,xxz}]dt+\sigma v^*_{i,xz} dB.
\end{eqnarray}
From (\ref{adjointeq1rst2}) it is clear that the partial (weak) derivative of the integral term with the respect to $z$ is zero (because it does not depend on $z$).
We differentiate  (\ref{adjointeq1rst2}) to get
$$v^*_{i,zt}+\partial_z\left[ v^*_{i,z} {H}^*_i(t, x, \frac{v^*_{i,x}}{v^*_{i,z}}, m)\right]+\frac{1}{2}\sigma^2 v^*_{i,zxx}=0,$$
which means that the drift term  is
$v^*_{i,tz}+v^*_{i,xz} \bar{b}+v^*_{i,zz} f+\frac{1}{2}\sigma^2v^*_{i,zxx}=0.$
Thus, $ d\eta^*_i= \sigma v^*_{i,xz} dB.$
Lets compute the diffusion coefficient more explicitly:
\begin{eqnarray}\sigma v^*_{i,xz} &=&
\frac{v^*_{i,z}}{v^*_{i,z}}\sigma v^*_{i,xz} 
=\eta^*_i l_i^*. 
\end{eqnarray}
Hence, one gets
$$d\eta^*_i= \eta^*_i l^*_i dB,\ \ \ \  \eta^*_i(T)=\theta_i e^{\theta_i[z_i(T)+h_i(x(T), m(T))]}.
$$
\end{proof}
\begin{proposition} \label{propeta} The function $(\eta^*_1,\ldots, \eta^*_n) $ solves the partial differential equation:
\be 0=\eta^*_{i,t}+\eta^*_{i,x} \bar{b}+\eta^*_{i,z} f_i+\frac{1}{2}\sigma^2\eta^*_{i,xx}\ee whenever these  derivatives  make sense. Moreover $\eta^*_i$ has a constant sign and has the same sign as $\theta_i.$
\end{proposition}
\begin{proof}
This follows from a weak derivative with the respect to $z_i$ in  Eq. (\ref{adjointeq1rst2}).
\end{proof}
Note that the function $m$ in Proposition $\ref{propeta}$ is the marginal of $\mu$ with the respect to $x$ and $\mu$ solves the Fokker-Planck-Kolmogorov forward equation with drifts $(\bar{b}, f)$ and diffusion coefficient $(\sigma, 0):$
\be  \label{fpk}
\mu_{t}+\partial_x[ \bar{b} \mu] +\partial_z(f \mu)-\frac{1}{2}\partial_{xx}(\sigma^2 \mu)=0,\ \ \ \
\ee
$\mu(0,dx,dz)=m_0(dx)\delta_0(dz).$
In view of Proposition \ref{lemmatx1}, 
$
(\bar{p}_i,\bar{q}_i, v^{\theta}_i, \ell_i)=(-p^*_i, q^*_i, \frac{\eta^*_i}{\theta_i}, \theta_i l^*_i)$ solves the  risk-sensitive SMP  (\ref{rs-firstAD-1})
and 
\begin{eqnarray}
dp^*_i&=&-\left[ H^*_{i,x}  +\frac{1}{\eta^*_i} E [ \eta^*_i \partial_x H^*_{i,m}] +  \sigma_x(\frac{\sigma v^*_{i,xx}}{\eta^*_i})\right]dt\\ && \nonumber
 -\left[ - p^*_i( l^*_i)^2+\frac{\sigma l^*_i v^*_{i,xx}}{\eta^*_i}\right]dt + (- p^*_i l^*_i+\frac{\sigma v^*_{i,xx}}{\eta^*}) dB,
\end{eqnarray}

\begin{eqnarray}p^*_i(T)=h_{i,x}(x(T),m(T)) + \nonumber \\ \nonumber
\frac{\tilde{E}\left[  e^{\theta_i (\tilde{z}_i(T)+h(\tilde{x}(T),m(T)))}\ \partial_x h_{i,m}(\tilde{x}(T),m(T))(x(T)) \right]}{e^{\theta_i ({z}_i(T)+h_i(x(T),m(T)))}}
\end{eqnarray}

The function $v^*(t,.)=\partial_{\mu}V$ is not the value function in the sense of Bellman because of the presence of the term $E[H_{i,m}]$ in Eq. (\ref{adjointeq1rst2}). $v^*(t,.)$ is the adjoint function (dual function) associated to mean-field-type best-response problem. Interestingly, in the mean-field free case, i.e., when $h_{i,m}=0, f_{i,m}=0, b_{i,m}=0,$ the dual function
 $(v^*_1(t,.),\ldots, v^*_n(t,.))$  coincides with the best-response value function of the risk-sensitive game problem with augmented state $(x,z).$  

{\color{black}
\section{Virus Spread over an evolving  network} \label{sec:virus}
WiFi network security has gained significant attention in research and industrial communities as a result of the global connectivity provided by the Internet. This has led to a variety of traditional defense mechanisms ranging from cryptography, firewalls, antivirus software, to intrusion detection systems.  Table  1  displays a sample
number of network attacks by major geographic region (State
or Country) with more  than 100 000 attacks. See \cite{onlineref} for more details on real time web attack monitoring.

  \begin{table*}  \caption{Internet attacks over the globe. An increase of around 13.02\%}  \label{state000fig0.}
\begin{tabular}{|c|l|l|l|l|l|l|l|l|}
\hline
\multicolumn{9}{|c|}{\textbf{Online Attacks}}                                   \\ \hline
\multicolumn{3}{|c|}{North America} & Europe  &  & Asia       &  & Australia &  \\ \hline
\multirow{7}{*}{}  & New York    & 617 150& Germany &  571255& India      &  2 353 001& Australia & 705594 \\ \cline{2-9} 
                   & Virginia    &  452916& Romania & 226018 & China      & 772447  &           &  \\ \cline{2-9} 
                   & Illinois    &  401766& Sweden  & 138543 & Bangladesh &  106102&           &  \\ \cline{2-9} 
                   & California  & 343627 &         &  &            &  &           &  \\ \cline{2-9} 
                   & Texas       & 309426 &         &  &            &  &           &  \\ \cline{2-9} 
                   & Ohio        & 110137 &         &  &            &  &           &  \\ \cline{2-9} 
                   & Florida     &  106779&         &  &            &  &           &  \\ \hline
\end{tabular}
\end{table*}

A virus that spreads through WiFi  networks as effectively as a human cold moves through cities, airports, public transport  areas has been explored recently. 
The virus can travel between WiFi networks via Access Points (APs) that connect households and businesses to WiFi networks. It can also propagate through femto cell and small cell networks.
We denote by $x$ the state of the entire network. $x$ could represent the number of access points that can be reached with infected relays (hotspots) at a specific period of the day. Since the number of access points that are active is highly stochastic and the number of nodes in the network is time-varying, $x$ is a random process. We do not consider a mass-action principle because there is no conservation of mass in this case, the population itself is random. It is unclear that the random process can be driven by Brownian but here we assume a small noise effect for simplicity. Users move over several geographical areas and some of them may carry portable wifi access points. Thus, the network is mobile and random. Each access point may interact with other hotspots in a certain neighborhood of communication and then the information/host propagates over multiple hops. In this setting an increasing rate is observed within the last decade. To capture this phenomenon, we propose an $L^{\alpha}-$norm drift model for the rate\cite{ref7}. The control parameter $u_1$  of the attacker may represent for example  the rate at which the virus  attempts transmission over the access points.
 \subsection{State dynamics models}
With explosive growth of mobile devices and the internet of things (IoT),  there is an increasing number of vulnerable devices and machines connected to these networks and hotspots so that the mass is not conserved. The population has a tendency the growth (in expectation). To Illustrate this we consider an information propagation model where the state is given by
$$ dx=\gamma(x) dt +\sigma dB, $$ where $\gamma(x)= \kappa x(1-\frac{x}{K}),$  $\kappa, K>0, $  $x(0)\in (0, K).$ 
\begin{figure}[h!]
\caption{State dynamics with three different noises.}
  \centering
    \includegraphics[width=0.9\textwidth]{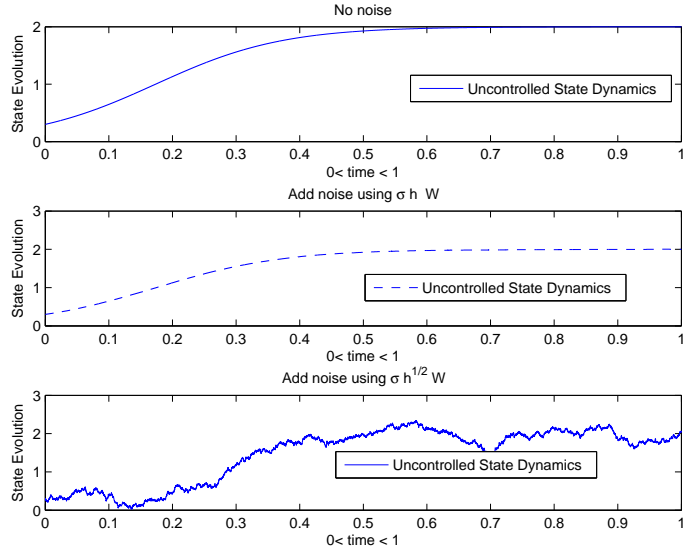}\label{state001fig1}
    
\end{figure}
Figure \ref{state001fig1} represents the evolution of the state under different noises. In Figure \ref{state001fig1} the parameters are $\kappa=10, K=2,$ and $\sigma\in \{0, 1\}$ and different noise terms are plotted. Starting from a initial state value $x(0)=0.3,$ we observe that the population state has tendency to move  around $2$ within the time interval $[0,1].$

\subsection{The state needs to be controlled}
Due to the presence of malicious attack in the network, there are lot of 
 security, privacy concerns so that the state needs to be controlled \cite{marriage}. To illustrate the model we introduce an attacker and a defender. Each of them has a control parameter, and has to make a certain decision on those parameters.
 $$ dx=[\gamma(x) +u_1-u_2] dt +\sigma dB, $$  $u_1$ is the attacker control strategy and $u_2$ is the defender control strategy.
 \begin{figure}[h!]
\caption{Effect of control effort $e=u_2-u_1=0.3$ in the state.}
  \centering
    \includegraphics[width=0.9\textwidth]{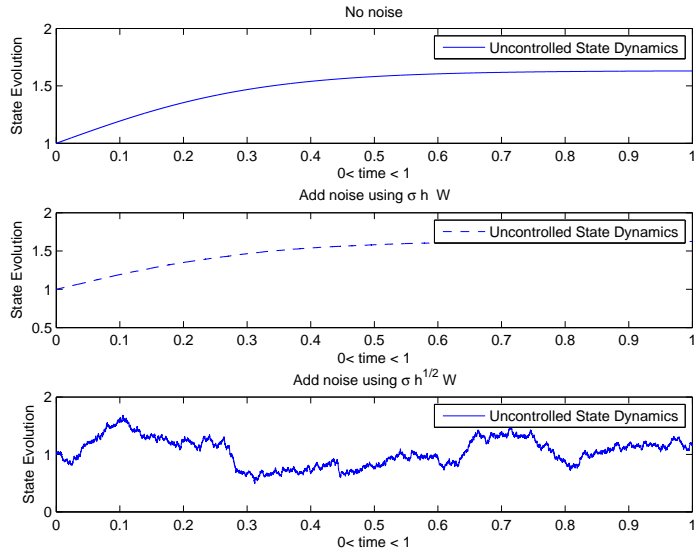}\label{state002fig2}
     
\end{figure}
  In Figure \ref{state002fig2} we represent the case where a significant control effort $e=u_2-u_1=0.3$ is injected into the system. We observe that the control affects significantly the state dynamics and help  towards a certain goal. For a significant effort $e=u_2-u_1=0.3$ invested into security, the state of infection can stay below the level $2$ starting from level $1.$ This means that control helps to reduce the infection rate and improve security.  However, the location of mobile devices and WiFi hotspots may be important in some cases, specially when local interaction and communications arise. In order to capture this phenomenon we introduce a non-linear behavior via the geographical location distribution and the intensity of interaction at time $t$ as  $m(t,dy)$  and we introduce 
  $b(t,x,y, u_1,u_2)= y[\gamma(x)+ u_1 -u_2],\ $  and the state dynamics becomes 
  $$ dx= |\gamma(x) +u_1-u_2| [\int_y y^{\alpha} m(t,dy)]^{\frac{1}{\alpha}} dt +\sigma dB. $$
  
  The variable $y$ can be seen as the intensity of interaction of infected devices/hosts. 
  \begin{figure}[h!]
  \caption{Effect of the mean-field term when the effort is $e=u_2-u_1=0.3$. Open-Loop case.}
  \centering
    \includegraphics[width=0.9\textwidth]{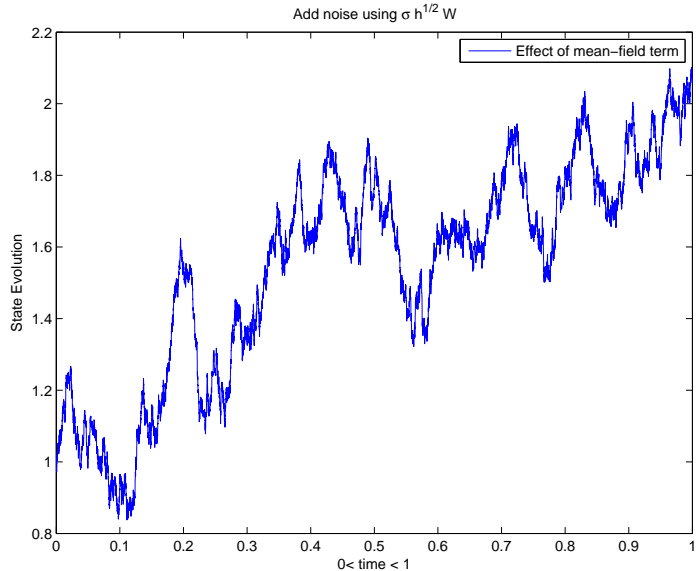}\label{state003fig3}
    
\end{figure}

 As we can see in Figure \ref{state003fig3},  the mean-field term $[\int_y y^{\alpha} m(t,dy)]^{\frac{1}{\alpha}}$ affects significantly the state dynamics in a multiplicative manner. The infection time increases rapidly with the mean-field term. Figure \ref{state004fig4} uses a state feedback strategy in the form of $[\int_y y^{\alpha} m(t,dy)]^{\frac{1}{\alpha}} x,$ with $\alpha=1.2.$    We observe that the infection state is significantly reduced compared to the open-loop case of Figure \ref{state003fig3}. This is illustrated in Figure \ref{state006} for several initial states.  Hence, it is important to strategically control the mean-field term so that the infected machines remains limited and the damage minimized. In order to do such a minimization we introduce below some objective functions.
 
 \begin{figure}[h!]
   \caption{Effect of the mean-field term under state and mean-field feedback strategies.}
  \centering
    \includegraphics[width=0.9\textwidth]{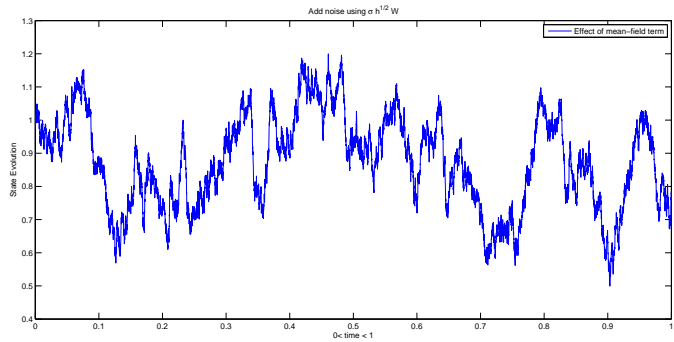}\label{state004fig4}
   
\end{figure}

 \begin{figure}[h!]
   \caption{Feedback strategies help to control and maintain the state below a certain range with high probability.}
  \centering
    \includegraphics[width=0.9\textwidth]{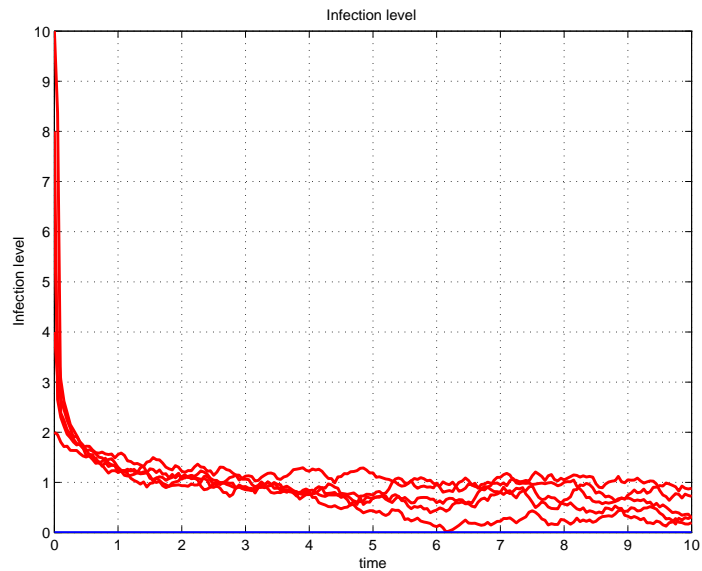}\label{state006}
\end{figure}

 \subsection{Objectives}
In the context of delay/disruption tolerant networks it can be shown that the delay and the probability of receiving the information have a natural risk-sensitive structure via Poisson arrival rates. However, the attacker and the network defense may not have the same sensitivity  when facing the risk. We denote by $\theta_1=\theta_a$ the attacker risk-sensitivity index and by $\theta_2=\theta_d$ the defender risk index. The cost of the attacker is  $f_1(x,m,u_1, u_2)=\frac{1}{2}u_1^{2}$ and  its terminal benefit (opposite signed for minimization) is $h_{1}= -\frac{c_1}{ \alpha}x^{ \alpha}-\frac{\bar{c}_1}{ \alpha}m_{ \alpha},$   where $m_{ \alpha}$ denotes the $ \alpha-$th moment of the state process. The goal for the attacker is then to find a tradeoff between the attack effort  cost $-\frac{1}{2}u_1^{2 }$ and the benefit $\frac{c_1}{ \alpha}x^{ \alpha}+\frac{\bar{c}_1}{ \alpha}m_{ \alpha}.$

\begin{figure}[h!]
  \caption{A typical one step cost function.}
  \centering
    \includegraphics[width=0.9\textwidth]{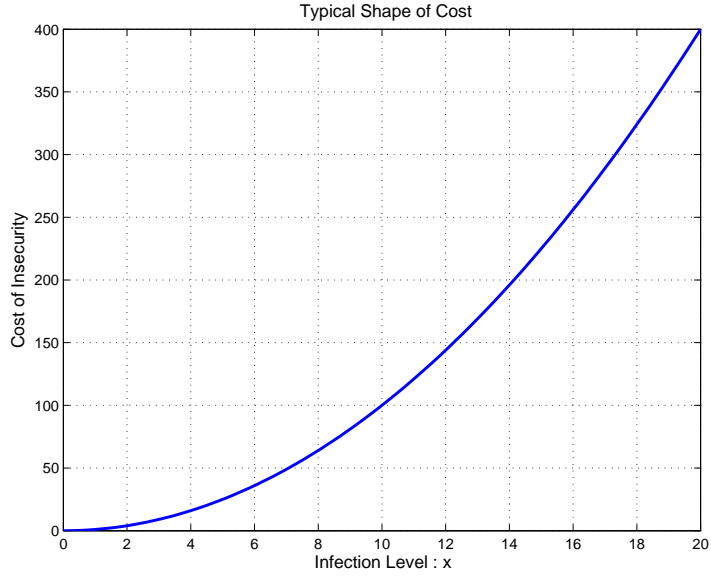}\label{shapecosttemeps}
    
\end{figure}

The  cost of the defender (could be the system administrator) is decomposed into damage cost and security investment loss  is  $ f_2= \frac{1}{2} x^2 +\frac{1}{2}u_2^2,$ $h_{2}= \frac{c_1}{ \alpha}x^{ \alpha}+\frac{\bar{c}_1}{ \alpha}m_{\alpha}, c_1\geq 0,\bar{c}_1\geq 0.$ For simplicity, the drift is chosen as $\bar{b}=\| b \|_{ \alpha}$ where $b(t,x,y, u_1,u_2)= y[\gamma(x)+ u_1 -u_2],\ \ \gamma(x)\geq 1.$ The control variables are limited to the interval $[0,1]$ at any time and the diffusion coefficient is  system-size dependent: $\sigma_{n(t)}:=\frac{\sigma}{n(t)}$ where $n(t)$ is a random variable representing the (active) system size at $t.$ Since there are multiple defense strategies, here we do pull them together in a cooperative manner as an ideal target. However, as observed in practice, the defender  may not coordinate their defense strategies due to non-alignment of objectives and/or professional privacy issues. Figure \ref{shapecosttemeps} represents a typical instantatenous cost. In Figure \ref{state005fig5} we plotted the evolution of a typical cost (random) invested into security over time.

\begin{figure}[h!]
  \caption{Evolution of the cost invested into security.}
  \centering
    \includegraphics[width=0.9\textwidth]{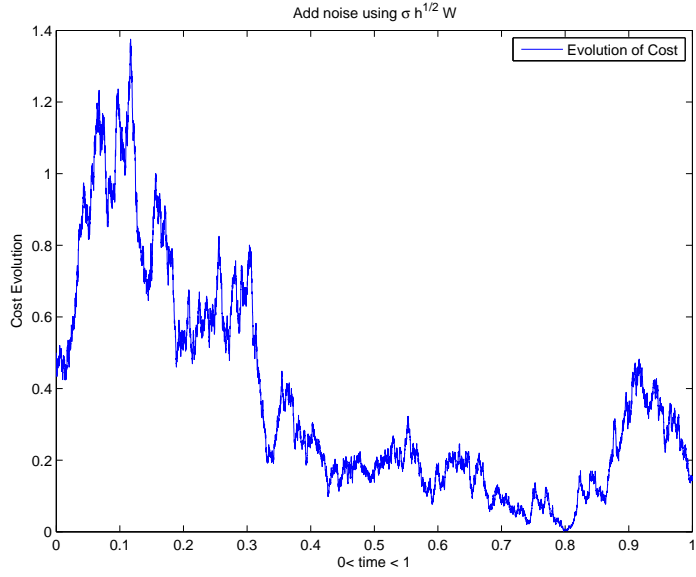}\label{state005fig5}
    
\end{figure}

These functions are not bounded, and we cannot use directly the existence results established above. However,  we provide the optimality equation for the interior case  and derive  a risk-sensitive SMP. 

$H_i=m_{\alpha}^{1/\alpha} [\gamma(x)+ u_1 -u_2] p_i + f_i.
$

The  attacker's optimal strategy  is
$$
  u_1 = [-m_{\alpha}^{1/\alpha} p_1]_0^1,
$$
where  $[a]_0^1:=\min (1, \max(0,a)).$ 
The defender's   optimal strategy  is
$$
u_2= [m_{\alpha}^{1/\alpha} p_2]_0^1,
$$
where $p_1, p_2, $ solve the  risk-sensitive SMP system :
$$dp_i=-\{ H_{i,x} +\frac{1}{v^{\theta}_i }E[v^{\theta}_i \partial_x H_{i,m}] \} dt+ q_i(-\theta_i l_i dt+dB).$$
where 
$H_{i,x}=m_{\alpha}^{1/\alpha} \gamma'(x)p_i+f_{i,x},\ $ $ \ f_{1,x}=0, f_{2,x}=x,$
 $H_{i,m}(., t, \xi,m)(x)= \frac{x^{ \alpha}}{ \alpha m_{ \alpha}^{ \alpha-1}} [\gamma(x)+ u_1 -u_2] p_i,$ and
$$\partial_x H_{i,m}(., t, \tilde{X},m)(x)=$$ $$ \frac{x^{ \alpha-1}}{m_{ \alpha}^{ \alpha-1}} [\gamma(x)+ u_1 -u_2] p_i+ \frac{x^{ \alpha}}{ \alpha m_{ \alpha}^{ \alpha-1}} \gamma'(x)p_i. $$
\subsection{Backward-Forward System}

\begin{equation}\label{SDEuBF}
\left\{\begin{array}{lll}
dp_1=-\{ H_{1,x} +\frac{1}{v^{\theta}_1 }E[v^{\theta}_1 \partial_x H_{1,m}] \} dt+ q_1(-\theta_1 l_1 dt+dB),\\
dp_2=-\{ H_{2,x} +\frac{1}{v^{\theta}_2 }E[v^{\theta}_2 \partial_x H_{2,m}] \} dt+ q_2(-\theta_2 l_2 dt+dB),\\
dx^u(t)=\bar{b}(.,t,x^u(t), m^u(t), u(t)) \ dt+ \sigma dB(t), \\ x^u(0)=x_0,\
m^u(t):=\mathcal{L}(x^u(t)),\
\end{array}\right.
\end{equation} where $v_i^{\theta},l_i$  and the terminal conditions solve   (\ref{rs-firstAD-1rsmaint2}) with the optimal strategies 
$$(u_1,u_2)=( [-m_{\alpha}^{1/\alpha} p_1]_0^1, [m_{\alpha}^{1/\alpha} p_2]_0^1).
$$
We investigate  (\ref{SDEuBF}) numerically  under stochastic Euler scheme (also called Euler-Maruyama scheme). We choose $\alpha=1.2.$ We set $\theta_1=0.1, \theta_2=0.3, c_1=0.8=\bar{c}_1.$The initial distribution is concentrated around two points: $1$ and $2.$ This can be extended to capture a geographical area where the attacks are concentrated in two main countries. In Figure \ref{figmain}, we plot the state distribution over time, the initial and final distribution, and the evolution of the expected values. We observe that the distribution moves progressively towards higher states. 

\begin{figure}
   \caption{Mean-field over time, the initial and final distribution, and the evolution of the expected values.}
  \centering
    \includegraphics[width=0.9\textwidth]{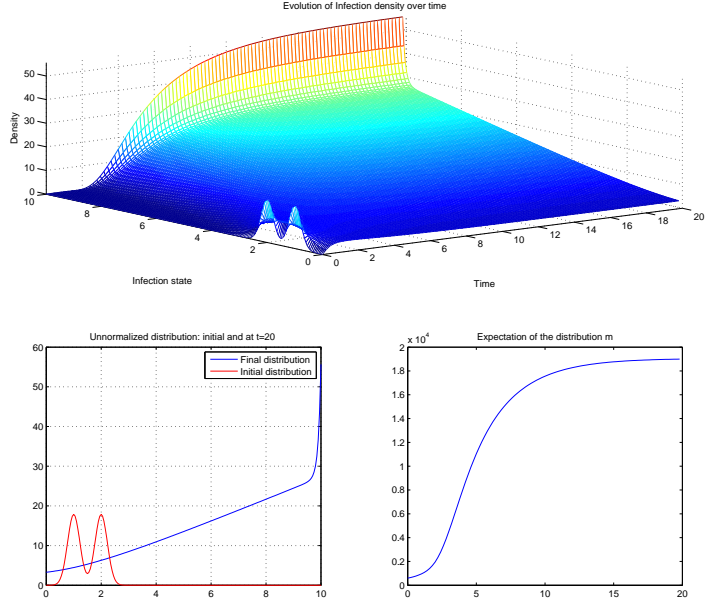}\label{figmain}
   
\end{figure}
}
\section{Concluding remarks}\label{sec:conclusion}
In this paper we have studied mean-field-type games with a drift that has $L^{\alpha}-$norm structure. Although this norm is not differentiable, it is possible to get existence of solutions. We have established relationship  between the risk-sensitive SMP and the dual functions. This allow us to verify the risk-sensitive SMP equations. When, the drift is mean-field free, we retrieve the classical risk-sensitive equations. The work can be extended in several ways. First, the interaction model in the drift $b(t,x,y,u)$ (which is pairwise interaction) can be modified to include $k-$wise interaction in the form $b(t,x,y_1,\ldots, y_k,u)$ with the measure $\prod_{l=1}^k \ m(t, dy_l). $ This is, in particular, useful for the control of virus spread over network where the interaction involves multiple nodes at a time. Second, the explicit solutions or qualitative analysis of SMP need to be conducted. Third, when $\alpha <1,$ we do not have a norm and the triangular inequality does not hold. In that case, quasi-norm type of inequalities need to established. We leave these open issues for future research.

\bibliographystyle{plain}

\appendix

\section*{Appendix}

\section{Mean-field convergence}
The extension of (i) the law of large numbers, (ii) central limit theorem and (iii) large deviation principle, from independent random variables  to sequences of indistinguishable random variables
has drawn the attention of a number of researchers ever since the appearance in Blum, Chernoff and co-authors \cite{blum}. Below we present some well-known results  and explain how they can be used in the McKean-Vlasov context with the $L^{\alpha}-$norm.

\subsection{Indistinguishability}

The notion of indistinguishability (or exchangeability or interchangeability) is introduced in order to discuss  
the existence of a limiting measure and mean-field convergence of the empirical measure of virtual particle states in the framework of 
de Finetti-Hewitt-Savage \cite{definetti,hewittsavage,aldous,kurtz,tra}.

Let $\mathcal{X}$ be a separable complete and metrizable topological space (Polish space).
 \begin{definition}[Indistinguishability] \label{defiindist}  
 A collection  $(x_{(1)},x_{(2)},\ldots, x_{(n)})$ of $\mathcal{X}-$valued random variables/processes, is indistinguishable (or exchangeable) if the joint law  
 is invariant by permutation over the index set $\{1,\ldots,n\},$ i.e.,
for any permutation $\sigma$ over the set  $\{1, 2, \ldots, n\}$, one has
\begin{eqnarray} \label{eqindistin}
\mathcal{L}(x_{(1)},x_{(2)},\ldots, x_{(n)})=\mathcal{L}(x_{\sigma(1)},\ldots,x_{\sigma(n)}),\
\end{eqnarray} where
$\mathcal{L}(X)$ denotes the law of the random variable $X.$ An infinite family of random variables/processes $(x^{(1)},x^{(2)},\ldots )$ is indistinguishable if every finite $n,$ the family $(x_{(1)},x_{(2)},\ldots, x_{(n)})$  is indistinguishable.
 \end{definition}

This says that the order (position) of the random variable in the family does not change the joint distribution. From  (\ref{eqindistin}) we also have that,
 for any measurable operator $O,$

\begin{eqnarray} \nonumber
\mathcal{L}\left(x_{(1)},x_{(2)},\ldots, x_{(n)},\  O(\frac{1}{n}\sum_{i=1}^n \delta_{x_{(i)}})\right)\\ \label{eqindistin3}
=\mathcal{L}\left(x_{\sigma(1)},\ldots,x_{\sigma(n)}, O( \frac{1}{n}\sum_{i=1}^n \delta_{x_{(i)}})\right),\
\end{eqnarray} where we do not permute the last component.

For indistinguishable  random variables/ processes, the convergence of the empirical measure $m_n:=\frac{1}{n}\sum_{i=1}^n\delta_{x_{(i)}}$ has been widely studied. This sits at the intersection between group theory and probability theory. The symmetry group properties have been used to derive some properties of the distributions of the processes.
The {\it de Finetti-Hewitt-Savage} theory 
provides the mean-field convergence of such a measure-valued process. 
When studying convergence of measures, an important issue is the choice of probability metric.
In order to measure the gap between two probability measures, we introduce the Wasserstein (Vasershtein) metric (also called Monge-Kantorovich metric) $d_{\alpha}$ of order ${\alpha}\geq 1.$

\begin{definition}[Wasserstein] 
$$d_{\alpha}^{\alpha}(\mu, \nu)=\inf\left\{     \int_{(x,y)} d_0( x,y)^{\alpha} \gamma(dx,dy);\right. $$ $$
\left. \gamma\in \mathcal{P}(\mathcal{X}\times \mathcal{X}),\ \gamma_x=\mu,\ \gamma_y=\nu\right\},$$
where $\gamma_x$ denotes the marginal with the respect to the $x-$component, where $d_0$ is a reference metric on $\mathcal{X}$ (such a metric exists because $\mathcal{X}$ is assumed to be metrizable). 
\end{definition}

The famous Kantorovich-Rubinstein 1958 theorem gives a dual representation of $d_1$ in terms of a Lipschitz-Bounded metric:
$$d(\mu, \nu):=d_1(\mu, \nu)=\sup\left\{   \int \phi d(\mu -\nu);\ \| \phi\|_{Lip}\leq 1\right\},$$
where $\| \phi\|_{Lip}=\| \phi\|_{\infty} +\sup_{x\neq y} \frac{| \phi(x)-\phi(y)|}{ d_0(x,y)}$ is the Lipschitz-norm of $\phi.$ 
It can be shown that $d_{\alpha}$ is a metric  (a "true" distance in a topological sense), i.e., it satisfies the axioms of a metric. 
For Polish spaces $\mathcal{X}$,  the Wasserstein distance $d_1$ is known to metrize the weak topology over $\mathcal{X}.$
As stated in Villani's book \cite{villani} the Wasserstein distance has the  following properties:
for any $1\leq {\alpha} < +\infty,$  $$\lim_n d_{\alpha}(m_n, m)=0$$ implies, in particular, that
\begin{itemize}
\item $ m_n$ converges to $ m$ in distribution  (weak convergence of probability measures) i.e., $$E_{m_n}[\phi] :=\int \phi dm_n\rightarrow  E_{m}[\phi],$$ as $n \rightarrow +\infty,$ for any measurable  bounded and Lipschitz functions $\phi.$ 
\item $ \int d_0^{\alpha}(x,y)m_n(dy) < +\infty$ for some $x\in \mathcal{X}.$
\end{itemize}

Thanks to these nice properties, the Wasserstein distance $d_{\alpha}$ is an appropriate candidate for the convergence of the empirical measure in the weak sense.

\begin{theorem}[de Finetti-Hewitt-Savage] \label{theoremsavage}
Let $x_{(1)},x_{(2)},\ldots,$ be an indistinguishable sequence of $\mathcal{X}-$valued
random variables, where $\mathcal{X}$ is a Polish space. Then, there is a $\mathcal{P}(\mathcal{X})-$valued 
random measure $m$ such that
$$
m=\lim_{m\rightarrow\infty} \ \frac{1}{n}\sum_{i=1}^n \delta_{x_{(i)}}, \ \mbox{almost surely}, 
$$
where $\mathcal{P}(\mathcal{X})$ denotes the space of probability measures on $\mathcal{X}.$  Conditioned on  $m$, the random variables $x_{(1)},x_{(2)},\ldots $ are $i.i.d$  with distribution $m$, that is, for each measurable bounded function $\phi,$
$$
\mathbb{E}\left( \phi(x_{(1)}, x_{(2)},\ldots , x_{(k)}) \ | \ m \right) $$  $$=\int \phi(y^1,\ldots, y^k) m(dy^1)\ldots m(dy^k).
$$
In addition, if the  moments of $x_{(i)}$ are finite then 
$$
d_1\left(m, \frac{1}{n}\sum_{i=1}^n \delta_{x_{(i)}}
\right)\leq \frac{C_1}{\sqrt{n}}=O\left (\frac{1}{\sqrt{n}}\right),$$ where $C_1 >0$ and $d=d_1$ denotes the Wasserstein metric of order one. 

%%%%%%%%%%%%%%%%%%%%%%%%%%%%%%%%%%%%%%%%%%%%%%%%%%%%%%%%%%%%%%%%%%%%%%%%%%%%%%%%%%%%%%%%%%%%%%%%%%%%%%
\end{theorem}

Note that the convergence in Theorem \ref{theoremsavage} is in the weak sense since the  Monge-Kantorovich distance $d_1$ metrizes the weak topology.
Theorem  \ref{theoremsavage} has been proved by de Finetti (1931, \cite{definetti}) for infinite binary sequences and has been extended by Hewitt and Savage (1955, \cite{hewittsavage})  to continuous and compact state spaces. A  simple and elegant proof can be found in Aldous (1985, \cite{aldous}), pp. 18-22, for the general state space. The rate of convergence for Monge-Kantorovich distance 
is obtained following the line of the law of large numbers of interacting systems. 
Theorem \ref{theoremsavage}  was  initially used  for static (time-independent) maps. Then, several applications in mathematical physics and biology, with dynamical models came into the picture. These are dynamically interacting particles, genes, molecules or nodes. Theorem \ref{theoremsavage}  was then extended to the dynamical case in at least two ways: (i)  path wise (up to a certain time step $T$),  (ii) at each time step $t.$

  \subsection{Large Deviation Principle for $m_n$}

We say that for any time $t,$ the probability measures $(m_n(t))_{n\geq 0}$ on a topological space obeys a Large Deviation Principle  with rate functions $(I(t,.))$ and in the scale  $(a_n)_n$ if  $(a_n)_n$ is a real-valued sequence satisfying $a_n \rightarrow \infty$ and $I$ is a non-negative, lower semicontinuous function such that
$$- \inf_{x\in int(B)} I(t,x) \leq  \liminf_n \frac{1}{a_n} \log m_n(t,B) $$ $$ \leq  \limsup_n \frac{1}{a_n} \log m_n(t,B) 
\leq  - \inf_{x\in cl({B})} I(t,x),$$  
for any measurable set $B$, whose interior is denoted by $int(B)$ and closure by $cl(B)$.  If the level sets $\{x : \ I(t,x) \leq  \beta \}$ are compact for every $\beta <+\infty$, $I(t,.)$ is called a {\it good rate function}. 
We introduce $\tilde{H}(.|.)$ as the  relative entropy function (defined also above)
$$
\tilde{H}(\mu | \ \nu):=\int \log(\frac{d\mu}{d\nu})\ d \mu,
$$ if $\mu$ is absolutely continuous with the respect to $\nu$ and $+\infty$ otherwise.

The main advantage of having this type of result is  the decay of $m_n(t,B)$ as $n$ gets large.
Basically, when the two limits are identical,
$m_n(t,B)$ is the order of $e^{-a_n R}$ where $R=\inf_{x\in B} I(t,x) > 0.$
As a consequence, the weak convergence from Theorem \ref{theoremsavage} and central limit  theorems can be derived  from these inequalities. 

The next result provides a large deviation principle result \cite{villani}.
\begin{theorem} 
Assume that  initially $m_{n}(0)$ follows a large deviation principle with rate $I(0, m(0))$ on the set of probability measures $\mathcal{P}(\mathcal{X}).$ Then
$(m_n(t,.))_{t\in [0,T]}$ follows a large deviation principle on the set of cadlag functions  from $[0,T]$ to $\mathcal{P}(\mathcal{X})$ with good rate
$$
I(t,m(t))=$$ $$\left\{\begin{array}{c}
I(0,m({0}))+\int_0^1 \tilde{H}(\ \dot{m}(s) \ | \ m(0)) \ \mbox{if} \ \dot{m}(s) \ \mbox{a.c.}\\
+\infty \ \mbox{otherwise}
\end{array} \right.
$$
where $a.c.$ means  absolutely continuous.
\end{theorem}

\begin{definition} 
Consider two processes $(x(t))_{t\in [0,T]}$ and $(y(t))_{t\in [0,T]}$ and set 
$$D_{T,{\alpha}}(\mu, \nu):=\inf\left\{      \left(E[\sup_{t\in [0,T]}  \ d_0( x(t),y(t))^{\alpha} ]\right)^{1/{\alpha}};\ \right. $$  $$\left. \mathcal{L}(x)=\mu,\ \mathcal{L}(y)=\nu\right\}.$$
\end{definition}
We now prove  Theorem \ref{cond1}  in several steps:
\subsection{Existence of solution to the state equation}
We start with the existence of solution. To prove existence of a solution with the respect to the Wasserstein distance, we adopt a contraction-type of approach. Then, we construct a Cauchy sequence ($L^{\alpha}$ space which is a Polish space). Consequently, the solution is almost unique. 
Consider the SDE given by
$$
x(t)= x(0) +\int_0^t \sigma(.,s,x(s)) dB(s) +$$ $$\int_0^t  \left[   \int_y |b|^{ \alpha}(., s, x(s), y,u)m(s, dy)\right]^{1/ \alpha}  \ ds =:RH(x)[m].
$$
Then one gets a  fixed-point stochastic equation  in $m$: $m= \mathcal{L}(x(t))=\mathcal{L}(RH(x)[m]).$  

Consider two measures $m_1(s, dy)$ and $ m_2(s, dy)$ such that  $D_{T, \alpha}(m_i, m_0) < +\infty$ then
$$
D_{t, \alpha}\left(   \mathcal{L}(RH(x)[m_1]), \mathcal{L}(RH(x)[m_2])  \right) $$ $$\leq 2L e^{2L t}\int_0^t  D_{s, \alpha}(m_1,m_2)\ ds,
$$
for any $t\in [0,T].$
Let $c>0, \ \Psi(f)(t)\leq c\int_0^t f(s)\ ds, \ t\in [0,T]$ for some operator $\Psi.$ and $M=\sup_{s\in [0,T]}\ f(s).$ 
 $$\Psi^2(f)(t)\leq c\int_0^t \Phi(f)(s)\ ds= c\int_0^t   [ c  M s]\ ds =c^2 \frac{t^2}{2} M.$$ 
Then, by induction, $\Phi^k(f)(t)\leq \left(  c^k \frac{T^k}{k!}\right) M,$ and
$$ D_{t, \alpha}\left(  \mathcal{L}^{k+1}(\mu_0), \mathcal{L}^{k}(\mu_0)  \right)\leq  C^k_T \frac{T^k}{k!}D_{T, \alpha}\left(  \mathcal{L}(\mu_0), \mu_0  \right) < +\infty,$$
where  $C_T= 2L e^{2L T},$ and where we have used the Lipschitz continuity  $L$ of $b$ and the Minkowski inequality for $p\geq 1.$
By summing up over natural numbers  $k,$  one gets $\sum_{k\geq k_0}D_{t, \alpha}\left(  \mathcal{L}^{k+1}(\mu_0), \mathcal{L}^{k}(\mu_0)  \right) < +\infty.$ 
Thus,  $(\mathcal{L}^{k}(\mu_0))_{k\geq k_0},\ $ is a Cauchy sequence with the respect to the metric $D_{T, \alpha},$ for $k_0\geq 1.$ Since we are in a complete metric space, this sequence converges to some fixed-point (say, $m$).
 Then $m=\mathcal{L}(RH(x)[m] )=\mathcal{L}(m)$  is the unique fixed-point, solution of the SDE.

\subsection{Pathwise mean-field convergence}
Consider $n$ independent random process $x^*_{i,n}$ satisfying 
$$
x^*_{i,n}(t)= x^*_{i,n}(0) +\int_0^t \sigma(.,s,x^*_{i,n}(s)) dB(s) $$ $$+\int_0^t  \left[   \int_y |b|^{ \alpha}(., s, x^*_{i,n}(s), y,u)m^*(s, dy)\right]^{1/ \alpha}  \ ds,
$$
$m^*(t,.)= \mathcal{L}(x^*(t)).$
and the particle representation 
$$
x_{i,n}(t)= x_{i,n}(0) +\int_0^t \sigma(.,s,x_{i,n}(s)) dB(s) $$ $$+\int_0^t  \left[   \int_y |b|^{ \alpha}(., s, x_{i,n}(s), y,u)m_n(s, dy)\right]^{1/ \alpha}  \ ds ,
$$
where 
$
m^*_n(t, .)=\frac{1}{n} \sum_{i=1}^n  \ \delta_{x^*_{i,n}(t)}.
$

It suffices to prove the statement for the coefficient with mean-field term. We take the difference between the two drift terms.
\begin{eqnarray}
D&=&\int_0^t  \left[   \int_y |b|^{ \alpha}(., s, x^*_{i,n}(s), y,u)m^*(s, dy)\right]^{1/ \alpha}  \ ds \\  &-& \int_0^t  \left[   \int_y |b|^{ \alpha}(., s, x_{i,n}(s), y,u)m_n(s, dy)\right]^{1/ \alpha}  \ ds 
\end{eqnarray}
We decompose $D$ into three separate terms as follows:
\begin{eqnarray}\nonumber 
D 
&{=}&\int_0^t  \left[   \int_y |b|^{ \alpha}(., s, x^*_{i,n}(s), y,u)m^*(s, dy)\right]^{1/ \alpha}  \ ds\\   \label{i1} &-& \int_0^t  \left[   \int_y |b|^{ \alpha}(., s, x^*_{i,n}(s), y,u)m^*_n(s, dy)\right]^{1/ \alpha}  \ ds \\   \nonumber 
&+& \int_0^t  \left[   \int_y |b|^{ \alpha}(., s, x^*_{i,n}(s), y,u)m^*_n(s, dy)\right]^{1/ \alpha}  \ ds \\  \label{i2} &-&  \int_0^t  \left[   \int_y |b|^{ \alpha}(., s, x^*_{i,n}(s), y,u)m_n(s, dy)\right]^{1/ \alpha}  \ ds \\  \nonumber
&+& \int_0^t  \left[   \int_y |b|^{ \alpha}(., s, x^*_{i,n}(s), y,u)m_n(s, dy)\right]^{1/ \alpha}  \ ds\\    \label{i3} &-&  \int_0^t  \left[   \int_y |b|^{ \alpha}(., s, x_{i,n}(s), y,u)m_n(s, dy)\right]^{1/ \alpha}  \ ds \\
D&=& I_1+I_2 +I_3
\end{eqnarray}
The first term $I_1$  (see (\ref{i1})) deals only with i.i.d random variables. Therefore, the convergence for that part is classical. For the second term $I_2,$ we use 
the triangular inequality for $ \alpha\geq 1.$ By Lipschitz continuity of $b$ , we get
        $ | I_2|\leq \frac{L}{n^{\frac{1}{ \alpha}}} \int_0^t  \left[   \sum_{i=1}^n \|x^*_{i,n}(s)-x_{i,n}(s)| \|^{ \alpha} \right]^{1/ \alpha}\ ds .$
        By Lipschitz continuity of $b,$
$ | I_3 |\leq  \int_0^t  |x^*_{i,n}(s`)-x_{i,n}(s)| \ ds.$

 By H\"{o}lder inequality
 \begin{eqnarray} \label{i4} 
   && \left\{E[\sup_{s\in [0,T]} |x^*_{i,n}(s)-x_{i,n}(s)|^{ \alpha}] \right\}^{1/ \alpha} \\ \nonumber
   &\leq&  
 \int_0^T \left\{ E\| Y_{n,s}  \|^{ \alpha}\right\}^{1/ \alpha}\ ds
 \\ \nonumber
  &+& \frac{L}{n^{\frac{1}{ \alpha}}} \int_0^T \left[   E\sup\|\sum_{j=1}^n |x^*_{j,n}(s)-x_{j,n}(s)| \|^{ \alpha} \right]^{1/ \alpha}ds \\ \nonumber
  && \nonumber +
  L \int_0^T \{ E[\sup_{s\in [0,t]} |x^*_{i,n}(s)-x_{i,n}(s)|^{ \alpha}]\}^{1/ \alpha}\ dt
\end{eqnarray} 
where \begin{eqnarray} \nonumber
Y_{n,s}&=&
    \left[\int_y |b|^{ \alpha}(., s, x^*_{i,n}(s), y,u)m^*(s, dy)\right]^{1/ \alpha} \\ && -    \left[\int_y |b|^{ \alpha}(., s, x^*_{i,n}(s), y,u)m^*_n(s, dy)\right]^{1/ \alpha} 
\end{eqnarray}
Summing over $i$ and using Gronwall Lemma yields

 $$\{E[\sup_{s\in [0,T]} |x^*_{i,n}(s)-x_{i,n}(s)|^{ \alpha}] \}^{1/ \alpha}$$ $$\leq 2L e^{2LT} 
 \int_0^T \left\{ E\|     Y_{n,s}  \|^{ \alpha}\right\}^{1/ \alpha}\ ds.$$
 which provides a convergence rate of $\sqrt{n}$ since  $$\ \sup_n \sqrt{n}\left\{ E\sup_{s\in [0,T]}\|     Y_{n,s}  \|^{ \alpha}\right\}^{1/ \alpha} <+\infty.$$
 This result shows that when $ \alpha\geq 1,$ and when the initial distributions of the virtual particles are mutually independent, with same distribution  as $x(0),$ then
 the particle interaction model with fixed control $u$  soon destroys that independence through the empirical measure $m_n$. 
 But, for a given finite time $t$, when the number of particle becomes large,  the mean-field convergence implies that the distributions
  become approximately independent again conditioning on $m(t,.)$, so that independence is still retained. This is called {\it propagation-of-chaos}.
  Note that this result is limited to finite horizon. For long-term behavior one needs to study the asymptotics (infinite horizon in time) of the SDEs in order to 
  derive propagation (or non-propagation) of chaos property.

\vspace{1cm}

  \newpage

  {\bf Hamidou Tembine} received his M.S. degree in Applied Mathematics from Ecole Polytechnique in 2006 and his Ph.D. degree in Computer Science  from University of Avignon in 2009. His current research interests include evolutionary games, mean field stochastic games and applications.  In 2014, Tembine received the IEEE ComSoc Outstanding Young Researcher Award for his promising research activities for the benefit of the society. He was the recipient of 5 best paper awards in the applications of game theory. Tembine is a prolific researcher and holds several scientific publications including magazines, letters, journals and conferences. He is author of the book on distributed strategic learning for engineers (published at CRC Press, Taylor \& Francis 2012), and co-author of the book Game Theory and Learning in Wireless Networks (Elsevier Academic Press). Tembine has been co-organizer of several scientific meetings on game theory in networking, wireless communications and smart energy systems.   He is a senior member of IEEE.

\end{document}